\documentclass[11pt,reqno]{article}
\usepackage[margin=1in]{geometry}
\usepackage{amsmath, amsthm, amssymb}
\usepackage[colorlinks=true, pdfstartview=FitV, linkcolor=blue,citecolor=blue, urlcolor=blue]{hyperref}
\usepackage[abbrev,lite,nobysame]{amsrefs}
\usepackage{times}
\usepackage[usenames,dvipsnames]{color}
\usepackage{mathtools}
\usepackage{bbold}

\usepackage{graphicx}
\usepackage{xcolor}

\newcommand{\norm}[1]{\left\|#1\right\|}

\newcommand{\abs}[1]{\left\vert#1\right\vert}

\renewcommand{\tilde}{\widetilde}
\def\eps{{\varepsilon}}
\newcommand{\les}{\lesssim}

\def\RR {\mathbb{R}}

\def\HH {\mathbb{H}}

\def\Bcal {{\mathcal B}}
\def\Ncal {{\mathcal N}}
\def\Lcal {{\mathcal L}}
\def\Mcal {{\mathcal M}}

\def\OO {{\mathcal O}}

\def\p{{\partial}}

\newcommand{\brak}[1]{{\left\langle #1 \right\rangle}} 

\newcommand{\ud}{\,\mathrm{d}}
\newcommand{\n}{\ensuremath{\nonumber}}

\newtheorem{theorem}{Theorem}[section]
\theoremstyle{lemma}

\theoremstyle{definition}

\newtheorem{remark}[theorem]{Remark}

\theoremstyle{lemma}
\newtheorem{lemma}[theorem]{Lemma}

\numberwithin{equation}{section}

\title{\bf Real analytic local well-posedness for the Triple Deck}
\author{
{\bf Sameer Iyer} 
\thanks{Department of Mathematics, Princeton University.
{\footnotesize \href{mailto:ssiyer@math.princeton.edu}{ssiyer@math.princeton.edu}.}
}
\and
{\bf Vlad Vicol}
\thanks{Courant Institute for Mathematical Sciences,  New York University.
{\footnotesize \href{mailto:vicol@cims.nyu.edu}{vicol@cims.nyu.edu}.}}
}
\date{}
\begin{document}
\maketitle

\begin{abstract} The Triple Deck model is a classical high order boundary layer model that has been proposed to describe flow regimes where the Prandtl theory is expected to fail. At first sight the model appears to lose two derivatives through the pressure-displacement relation which links pressure to the tangential slip. In order to overcome this, we split the Triple Deck system into two coupled equations: a Prandtl type system on $\mathbb{H}$ and a Benjamin-Ono type equation on $\mathbb{R}$. This splitting enables us to extract a crucial leading order cancellation at the top of the lower deck. We develop a functional framework to subsequently extend this cancellation into the interior of the lower deck, which enables us to prove the local well-posedness of the model in tangentially real analytic spaces.
\hfill {\bf \today}  
\end{abstract}

\setcounter{tocdepth}{1}
\tableofcontents

\section{Introduction}
A fundamental challenge in fluid mechanics is to describe the vanishing viscosity limit $(\nu \rightarrow 0)$ of the Navier-Stokes equations on domains with a solid boundary. In this paper we consider the fluid domain to be the two-dimensional half space $\mathbb{H}$.  The main difficulty is due to the incompatibility between the no-slip boundary condition for the Navier-Stokes velocity field ($\bold{U}^{\nu}|_{\p \Omega} = 0$) and the slip boundary condition for the Euler velocity field $(\bold{U}^E|_{\p \Omega} \cdot (0,-1)= 0)$, which makes it difficult to obtain uniform in $\nu$ estimates for norms of $\bold{U}^\nu$ which are stronger than $L^\infty_t L^2_x$. 

\subsection{Historical Overview}

In order to rectify this mismatch, Prandtl~\cite{Prandtl1904} proposed the existence of a thin, $\OO(\nu^{\frac{1}{2}})$, fluid layer near the boundary through which the Navier-Stokes velocity field transitions from an outer Euler flow in the bulk, to the no-slip condition on the solid wall. Mathematically, this corresponds to a formal asymptotic expansion of the viscous incompressible flow $\bold{U}^{\nu}$ as 
\begin{align} \label{Pr.ansatz}
\bold{U}^\nu(t, x,y) = \bold{U}^E(t, x,y) + \bold{U}^{BL}(t,x, \frac{y}{\sqrt{\nu}}) + \OO(\nu^{\frac{1}{2}})
\end{align}
where  $\bold{U}^E$ is the  Euler flow typically assumed to be known \textit{a-priori}, and $\bold{U}^{BL}$ is known as the Prandtl boundary layer corrector. The boundary layer unknown $\bold{u}^P = [u^P,v^P] := \bold{U}^{BL} + \bold{U}^E|_{Y = 0}$ is a function of the tangential variable $x$ and the normal fast variable $\bar Y = \frac{y}{\sqrt{\nu}}$, and is governed by the famous Prandtl boundary layer equations 
\begin{subequations} \label{Pr:sys}
\begin{align} \label{Pr:sys:a}
&\p_t u^P + u^P \p_x u^P + v^P \p_{\bar Y} u^P - \p_{\bar Y}^2 u^P =  - \p_x P^E(t,x,0) , \\ \label{Pr:sys:b}
&\p_x u^P + \p_{\bar Y} v^P = 0, \\ \label{Pr:sys:c}
&[u^P, v^P]|_{{\bar Y} = 0} = 0, \hspace{3 mm} u^P|_{{\bar Y} \to 0} = u^E(t,x,0), \hspace{3 mm}  
\end{align}
\end{subequations}
posed in the half space $\HH = \{ (x,\bar Y) \colon \bar Y > 0\}$. The system treats the Euler pressure trace $P^E$ and the Euler wall slip velocity $u^E$ as known, and is supplemented with an initial condition $u^P|_{t = 0} = u^P_0(x,\bar Y)$.

A first step towards establishing the validity or the invalidity of the Prandtl expansion \eqref{Pr.ansatz} is a detailed understanding of the Prandtl system \eqref{Pr:sys} itself.  The well- and ill-posedness of the Prandtl equations has a long history of which we only provide a very brief summary (see the reviews~\cite{BardosTiti13,MaekawaMazzucato16} for further references). Under the monotonicity assumption  $\p_y u^P|_{t = 0} > 0$, Oleinik~\cite{Oleinik66,OleinikSamokhin99} obtained global in time, regular solutions on the domain $[0, L] \times \mathbb{R}_+$ for small $L$, and local in time regular solutions for arbitrary finite $L$. The aforementioned results rely on the Crocco transform, which is available from the monotonicity hypothesis. See also the global in time existence result of weak solutions obtained in \cite{XinZhang04}  under the additional assumption of a favorable pressure gradient $\p_x P^E(t,x) \le 0$. Without using the Crocco transform, local existence was established in the works of \cite{MasmoudiWong15,KukavicaMasmoudiVicolWong14} using energy methods and \cite{AlexandreWangXuYang14} using a Nash-Moser iteration. When the monotonicity assumption is removed, local well-posedness results for \eqref{Pr:sys} were first established  assuming tangential real analyticity of the initial datum~\cite{SammartinoCaflisch98a,LombardoCannoneSammartino03,KukavicaVicol13} (see also~\cite{IgnatovaVicol16} for an almost global existence result for small datum), and more recently assuming only tangential Gevrey-class regularity~\cite{GerardVaretMasmoudi13,LiYang16}. The sharp Gevrey-$2$ result without any structural assumptions was recently established in~\cite{DiGV18}. On the other hand, in Sobolev spaces without monotonicity, the Prandtl equations are ill-posed, as was shown in~\cite{GerardVaretDormy10,GuoNguyen11,GerardVaretNguyen12,LiuYang17}.

Concerning the validity of Prandtl ansatz~\eqref{Pr.ansatz}, in the unsteady setting the expansion has been verified locally in time assuming the initial datum is real-analytic~\cite{SammartinoCaflisch98b,WangWangZhang17,NguyenNguyen18}, under the assumption that the initial vorticity is supported away from the boundary~\cite{Maekawa14,FeiTaoZhang16,FeiTaoZhang18}, in the Gevrey setting for initial data close to certain stable shear flows~\cite{GerardVaretMaekawaMasmoudi16}, or assuming that the initial vorticity is analytic only near the boundary of the half space~\cite{KukavicaVicolWang19}. 
In contrast, for initial datum in Sobolev spaces the ansatz \eqref{Pr.ansatz} has been proven to be {\em invalid}~\cite{Grenier00,GrenierGuoNguyen14b,GrenierGuoNguyen14c,GrenierNguyen17}, with the recent result~\cite{GrenierNguyen18a} proving that the expansion is not valid in the $L^\infty$ topology.

A notable success of the Prandtl theory is in the steady regime, where it was in fact derived in~\cite{Prandtl1904}. For steady flows  in~\eqref{Pr:sys:a},   the initial datum~\eqref{Pr:sys:c} is typically replaced by in-flow data at $\{x = 0\}$, which represents for instance the leading edge of a flat plate. Shortly after Prandtl's original work,  Blasius~\cite{Blasius08} discovered the self-similar solution to the steady Prandtl equations 
\begin{subequations}
\label{eq:Blasius}
\begin{align} \label{Blasius.a}
&[u^P, v^P] := [f'(\eta), \frac{1}{\sqrt{x}} \{ \eta f'(\eta) - f(\eta)\} ],  \text{ where } \eta = \frac{y}{\sqrt{x}}, \\ \label{Blasius.b}
&f f'' + f''' = 0, \qquad f'(0) = 0, f'(\infty) = 1, \frac{f(\eta)}{\eta} \xrightarrow{\eta \rightarrow \infty} 1. 
\end{align}
\end{subequations}

\noindent Experiments have confirmed the accuracy of \eqref{Pr.ansatz} for steady flow over a plate to a remarkable degree of precision~\cite{Schlichting60}, especially for the Blasius self-similar boundary layers \eqref{eq:Blasius}. Mathematically, in the steady case, the ansatz has been recently verified in \cite{GerardVaretMaekawa18} for shear boundary layer flows which arise from forced Navier-Stokes equations, and \cite{GuoIyer18,GuoIyer18a} for a general class of $x$-dependent boundary layer flows, which arise from homogeneous Navier-Stokes flows, and which include the Blasius solution. See also~\cite{GuoNguyen14} for related results on a moving plate. 

In spite of  the success of the Prandtl theory in the steady regime, the phenomenon of boundary layer separation  remains mostly unsolved, both in the steady and the unsteady regimes~\cite{Schlichting60,SychevEtAl98,CousteixMauss07,SmithBrown12}. In the unsteady case, the van Dommelen and Shen singularity~\cite{VanDommelenShen80}, which was recently proven to occur rigorously~\cite{EEngquist97,KukavicaVicolWang17,CollotGhoulMasmoudi18,CollotGhoulIbrahimMasmoudi18} may be seen as as a diagnostic of separation~\cite{GarganoSammartinoSciacca09}: an adverse Euler pressure gradient causes a finite time singularity in the displacement thickness, and so the flow is detached from the flat plate. The vorticity generated at the boundary is ejected into the bulk of the flow where it rolls up and is considered as one of the factors responsible for the anomalous dissipation of energy. In the steady case the detachment of the boundary layer from the flat plate was predicted by Goldstein~\cite{Goldstein48} and has been proven recently in~\cite{DalibardMasmoudi18}. This breakdown of the assumptions on which Prandtl equations are derived signals the limitations of the classical Prandtl boundary layer theory, and new, higher order, theories are required in order to model the inviscid-boundary layer coupling near points of separation \cite{CebeciCousteix05,CousteixMauss07}. 

Two well-known higher order models are the {\em Prescribed Displacement Thickness} (PDT) model~\cite{CatherallMangler66} and the {\em Interactive Boundary Layer} (IBL) model~\cite{Carter74,LeBalleur90,Lagree10}. For instance, in the IBL model the Euler flow and boundary layer flow are strongly coupled through a boundary condition of the type 
\begin{align} \label{eq:IBL}
v^E|_{y = 0} = \sqrt{\nu}\p_x \{ \kappa u^E|_{y = 0} \}, \qquad \kappa := \int_{\mathbb{R}_+} \Big(1 - \frac{u^P}{u^E|_{Y = 0}} \Big) \ud y, \qquad  u^P|_{y \to \infty} = u^E(t,x,0).  
\end{align}
This model has been studied rigorously in~\cite{DDLM18}, where it is shown to be  {\em linearly ill-posed even in analytic spaces}. Similar dramatic ill-posedness results are shown in \cite{DDLM18} to hold for the PDT model. These severe instabilities in the PDT and IBL higher order boundary layer models lead us to consider the Triple Deck system, which is the main purpose of this paper.

\subsection{Triple Deck equations}
In order to describe the Triple Deck system, it is useful to keep in mind the below diagram, taken from \cite[pp.~220, Figure 4]{Smith82}, which describes the steady flow past a finite plate whose boundary is at $\{y=0\}$, with a leading edge to the left and a trailing in the bottom center of the figure:
 \begin{center}
\includegraphics[width=0.6\textwidth]{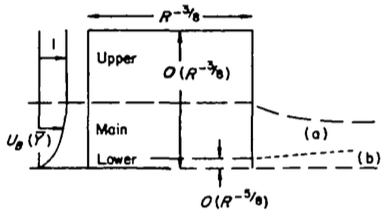}
\end{center}
Here $R$ denotes the Reynolds number. Near the leading edge of the plate, the flow is accurately described by the Prandtl theory, and in particular by the self-similar Blasius profile~\eqref{eq:Blasius}. The trailing edge of the plate creates a disturbance, and the flow undergoes the so-called Goldstein singularity. The triple-deck theory describes specifically the transition from the Blasius profile on the left of the plate to the Goldstein near which occurs after the trailing edge of the plate. This was formalized in the works of \cite{Stewartson68,Stewartson69,Neiland69,Messiter70} who proposed the {\em three deck} structure, and introduce the horizontal $\OO(\nu^{-\frac 5 8})$ and vertical $\OO(\nu^{-\frac 5 8}), \OO(\nu^{-\frac 3 8})$ length scales  that are not present in the Prandtl theory. The notion of introducing different scales at the point of boundary layer separation was   introduced earlier in~\cite{Lighthill53}.  We refer the reader to the works~\cite{Smith82,Meyer82,Klingenberg83,Meyer83,CowleyTutty85,Lagree} for an overview of the ideas and history behind the Triple Deck model, and include a formal derivation of the unsteady triple deck model (cf.~system \eqref{eq:TD:main}--\eqref{eq:TD:A:p} below) in Appendix~\ref{sec:appendix} of  this paper.

Specifically, see e.g.~\cite[Section 3]{Smith82},~\cite[Section 2]{CowleyTutty85} or~\cite[Section 4]{Duck87}, we consider the {\em unsteady  Triple Deck system} posed in the half space $\HH = \{ (x,y) \colon y> 0\}$, which is given by 
\begin{subequations}
\label{eq:TD:main}
\begin{align}
\partial_t u + u \partial_x u + v \partial_y u &= - \partial_x p + \partial_y^2 u \label{eq:TD:main:a}\\
\partial_x u + \partial_y v &=0 \label{eq:TD:main:b}\\
\partial_y p &= 0 \label{eq:TD:main:c}
\end{align}
\end{subequations}
supplemented with the boundary conditions 
\begin{subequations}
\label{eq:TD:BC}
\begin{align}  \label{eq:TD:BC:a}
&u(x,0,t)= v(x,0,t) = 0,  \\ \label{eq:TD:BC:b}
&u(x,y,t) - y \to 0  \qquad \mbox{as} \qquad x \to - \infty \\ \label{eq:TD:BC:c}
&u(x,y,t) - y \to A(x,t) \qquad \mbox{as} \qquad y \to + \infty  
\end{align}
\end{subequations}
and with the {\em pressure-displacement relation}
\begin{align}
\label{eq:TD:A:p}
p(x,t) = \frac{1}{\pi} p.v. \int_{\RR} \frac{(\partial_x A)(\bar x,t)}{x-\bar x} d\bar x =  |\partial_x| A(x,t)  
\end{align}
characteristic of incompressible flows. Note that other pressure-displacement relations may be specified in the case of supersonic and jet-like compressible boundary layer flows (see e.g.~\cite[Equation (2.4a-e)]{CowleyTutty85}). The system \eqref{eq:TD:main}--\eqref{eq:TD:A:p} is supplemented with a compatible initial condition
\begin{align}
\label{eq:TD:IC}
u|_{t=0} = u_0
\end{align}
on $\HH$.  Note that while equations \eqref{eq:TD:main} look the same as the classical Prandtl equations, the main difference is that $p$ is not given in advance, and neither is the value of $u$ at the top of the  lower deck. Instead, these are coupled by the relation \eqref{eq:TD:A:p} above.

The Triple Deck and the IBL models share the common feature that $u$, respectively $u^P$,  converge  as $y \to \infty$ to a function that is not given \textit{a-priori}, and must determined through the evolution. However, in contrast to \eqref{eq:IBL} in which $u^P|_{y \to \infty}$ is governed ultimately by the Euler equations, the behavior of $u|_{y \to \infty}$ in~\eqref{eq:TD:main} is governed by the Benjamin-Ono equations, as is shown in Section~\ref{sec:BO:evo}. It has been alluded to in~\cite{DDLM18} (see also~\cite{Smith79,Duck87}) that the Triple Deck has favorable stability features relative to the IBL model, but to our knowledge this has not been studied mathematically until the present work. In fact, it is not known whether the system \eqref{eq:TD:main}--\eqref{eq:TD:IC} is well-posed, even locally in time.

The unsteady Triple Deck model poses significant mathematical difficulties, because the map $u \mapsto \partial_x p$ loses two derivatives in $x$ (in view of \eqref{eq:TD:A:p}) and half a derivative in $y$ (due to the restriction to the boundary $\{y=\infty\}$). The {\em two derivative loss} in   $x$   seems to preclude the well-posedness of the system, even in spaces of analytic functions. Our goal is to show that due to a certain cancellation in $L^2_x$, the   loss is only of one derivative in $x$, and hence the system admits local in time real-analytic solutions with respect to $x$, which are   Sobolev smooth in $y$. Our main result is Theorem~\ref{thm:main} below, which  may be stated informally as: assume that $A_0(x)$ is real-analytic and that  the function  $u_0(x,y) - y - A_0(x)$ is  tangentially real-analytic and lies in a weighted $L^2$ space with respect to the normal variable; then there exits locally in time a unique solution in this class. We discuss the main difficulties and the main ideas of the  proof in Section~\ref{sec:main} below. Prior to this, we introduce the functional setting of the paper and the decomposition \eqref{eq:decomp:1}--\eqref{eq:decomp:2} of the solution.

\subsection{Main result and functional setting} 
 
\subsubsection{Analytic norms}
In order to measure decay in $y$, we introduce the $y$-weight given by 
\begin{align}
\rho (y,t) = e^{\frac{y^{2}}{8(1+t/\eps)}} \, . 
\label{eq:rho:def}
\end{align} 
Note that $\rho$ does not depend on $x$. The parameter, $\eps$, appearing in \eqref{eq:rho:def} will be selected small, based only on the initial datum, according to the relation \eqref{choice:eps:a}. The time scale over which we prove existence will be, without loss of generality, restricted to $t \in [0,T_\ast]$, where $T_\ast \le \eps$, so that in particular the quotient $t/\eps$ appearing in the weight above is bounded. 

We denote the Fourier transform a function $f$ in the $x$-variable only, at frequency $\xi \in \RR$, as $f_\xi= f_\xi(y,t)$. Since $f$ is real-valued, we automatically have that $f_{-\xi} = \overline{f_\xi}$. 
For $\tau = \tau(t)>0$, $r > 2$, and a function $f(x,y,t)$ we use Plancherel to define 

\begin{align} \label{eq:norm}
\norm{f}_{\tau,r}^2 &= \norm{\rho \, e^{\tau |\partial_x|} f}_{H^r_x L^2_y}^2 = \int_{\RR}  \int_0^\infty  \rho^2(y) \abs{f_\xi(y)}^2 e^{2\tau |\xi|} \brak{\xi}^{2r} \ud y  \ud \xi 
\end{align}
where $\brak{\xi}^2 = 1 + |\xi|^2$, and we have suppressed the time dependence of $\tau, \rho$, and $f_\xi$.
Associated to this norm, it is convenient to also define the inner product
\begin{align*} 
\brak{f,g}_{\tau,r} = \int_\RR \int_0^\infty  \rho^2(y)  f_\xi(y) \overline{g_\xi(y)} e^{2\tau |\xi|} \brak{\xi}^{2r} \ud y  \ud \xi 
\end{align*}
where the time dependence is suppressed. The idea to use real-analytic norms of the type~\eqref{eq:norm} goes back to the work of Foias-Temam~\cite{FoiasTemam89} in the context of the Navier-Stokes equations, and to~\cite{LevermoreOliver97} in the context of the Euler equations. See also~\cite{OliverTiti00,KukavicaVicol09,BonaGrujicKalisch10} and references therein.
 
Notice that by definition of the $\norm{\cdot}_{\tau,r}$ norm, we have the identity  
\begin{align}
\frac{1}{2} \frac{d}{dt} \norm{f}_{\tau,r}^2 +  (-\dot{\tau}) \norm{|\partial_x|^{1/2}f}_{\tau,r}^2 
&= \int_{\RR} \left( \frac 12 \frac{d}{dt} \norm{\rho  \, f_\xi}_{L^2_y([0,\infty))}^2 \right) e^{2\tau |\xi|} \brak{\xi}^{2r} \ud\xi      
\notag\\
&= \brak{\partial_t f + f \partial_t (\log \rho), f}_{\tau,r}  \, .
\label{eq:ODE:XY}
\end{align}
Therefore, a decrease in the analyticity radius yields a $|\partial_x|^{1/2}$-dissipative term.

We introduce similar analytic norms for functions $g(x,t)$, which are independent of $y$. Here, we let
\begin{align}
\label{eq:tilde:X:norms}
\norm{g}_{\tilde{\tau,r}}^2 = \norm{e^{\tau|\partial_x|} g}_{H^r_x} = \int_{\RR} |g_\xi|^2 e^{2\tau |\xi|} \brak{\xi}^{2r} \ud\xi  
\end{align} 
with associated inner product
\begin{align*}
\brak{f,g}_{\tilde{\tau,r}} = \int_\RR   f_\xi  \overline{g_\xi} e^{2\tau |\xi|} \brak{\xi}^{2r}  \ud \xi \,.
\end{align*}
As in \eqref{eq:ODE:XY}, we have
\begin{align}
\frac 12  \frac{d}{dt} \norm{g}_{\tilde{\tau,r}}^2  +  (-\dot{\tau}) \norm{|\partial_x|^{1/2} g}_{\tilde{\tau,r}}^2  = \int_{\RR} \left(\frac 12 \frac{d}{dt}  |g_\xi|^2   \right) e^{2\tau |\xi|} \brak{\xi}^{2r} = \brak{\partial_t g, g}_{\tilde{\tau,r}}\, .
  \label{eq:ODE:tilde:XY}
\end{align}
Having defined the basic norms, we turn to the definition of the total norms used in this paper and the corresponding unknowns that we measure using these norms.

\subsubsection{Representation of the solution and the total energy}
We shall work with the following decomposition of the solution $u(x,y,t)$ of \eqref{eq:TD:main}--\eqref{eq:TD:BC}.  We write
\begin{align}
u(x,y,t) = y + w(x,y,t)
\label{eq:decomp:1}
\end{align}
where the function $w(x,y,t)$ is defined in terms of its tangential (i.e., with respect to $x$) Fourier transform coefficients, $w_\xi(y,t) = \int_{\RR} w(x,y,t) e^{-i x \xi} dx$, given by 
\begin{align}
w_\xi(y,t) = \bar w_{\xi}(y,t) + A_\xi(t) \theta_\xi(y,t)
\, .
\label{eq:decomp:2}
\end{align}
The Gaussian weight function $\theta_\xi(y,t)$ is defined {\em explicitly} in \eqref{eq:theta:def} below.
The coefficients $A_{\xi}(t)$ are nothing but the Fourier coefficients in $x$ of the function $A(x,t) = \lim_{y\to \infty} w(x,y,t)$. In Section~\ref{sec:BO:evo} we show that $A$ obeys a {\em forced Benjamin-Ono equation}, cf.~\eqref{eq.A}, which arises as a compatibility equation for \eqref{eq:TD:main}--\eqref{eq:TD:BC}. On the other hand, the main unknowns $w_{\xi}(y,t)$ are shown in Section~\ref{sec:Prandtl:evo} to solve an evolution equation, cf.~\eqref{eq:bar:w:evo}--\eqref{eq:NLMB:def}, which has a very similar structure to the classical Prandtl system, with the addition of certain singular coupling terms to the evolution for $A$. The point is that the original function $u$ may be reconstructed explicitly from knowledge of the Fourier coefficients $\bar w_\xi$ and $A_\xi$. Accordingly, our total norms measure the analytic regularity of $\bar w$ and $A$. 

Throughout the paper fix a value for $r>2$ and a smooth cutoff function $\chi(y)$ approximating ${\bf 1}_{\{y\geq 2\}}$ (defined in \eqref{eq:chi:def} below). For a function $\tau(t)>0$ to be defined later, for a parameter $\delta > 1$ to be chosen precisely later, and with the norms $\norm{\cdot}_{\tau,r}$ and $\norm{\cdot}_{\tilde{\tau,r}}$ defined in \eqref{eq:norm} respectively \eqref{eq:tilde:X:norms}, we let
\begin{subequations}
\label{eq:cluster:fuck}
\begin{align} \label{X:tau:norm:def}
\norm{(\bar w,A)}_{X_\tau} 
&= \norm{\bar w(t)}_{\tau(t),r} + \frac{1}{\delta} \norm{\chi \p_y \bar w}_{\tau(t),r-1/2} +  \norm{A(t)}_{\tilde{\tau(t),r}} \\ \label{Y:tau:norm:def}
\norm{(\bar w,A)}_{Y_\tau} 
&=  \norm{\abs{\p_x}^{1/2}  \bar w(t)}_{\tau(t),r} + \frac{1}{\delta} \norm{\chi \abs{\p_x}^{1/2}  \p_y \bar w}_{\tau(t),r-1/2} +  \norm{\abs{\p_x}^{1/2}  A(t)}_{\tilde{\tau(t),r}}\\ \label{norm:Z:tau}
\norm{ \bar w }_{Z_\tau} 
&=  \norm{\p_y \bar w(t)}_{\tau(t),r} + \frac{1}{\delta} \norm{\chi  \p_{yy} \bar w}_{\tau(t),r-1/2} \\ \label{norm:H:tau}
\norm{ \bar w }_{H_\tau} 
&=  \norm{y \bar w(t)}_{\tau(t),r} + \frac{1}{\delta} \norm{y \chi  \p_{y} \bar w}_{\tau(t),r-1/2}
\, .
\end{align}
\end{subequations}
The $X_\tau$ norm is the main analytic-in-$x$ and weighted $L^2$-in-$y$ norm used in this paper. The $Y_\tau$ norm quantifies dissipation in the $x$ variable due to a shrinking  analyticity radius, the $Z_\tau$ norm quantifies dissipation in the $y$ variable due to the $\partial_{yy}$ terms present in the equation, while the $H_\tau$ norm encodes a gain of a $y$ weight which is important due to the unboundedness of the domain $[0,\infty)$.  Associated to these norms we   define the total analytic energy via 
\begin{align} \label{def:E}
E(T) &= \sup_{t \in [0, T]} \norm{(\bar w,A)}_{X_{\tau(t)}}^2 + \int_0^T \!   \| (\bar{w}, A) \|_{Y_\tau}^2 \ud t + \frac{1}{16} \int_0^T \!  \norm{\bar w}_{Z_{\tau(t)}}^2 \ud t + \frac{1}{64 \eps} \int_0^T \!   \norm{\bar w}_{H_{\tau(t)}}^2 \ud t.
\end{align}

\subsubsection{Main Theorem and Overview of Proof}
\label{sec:main} 

We are now ready to state the main result. 
\begin{theorem}[\bf Main Theorem] \label{thm:main} Fix an initial radius of analyticity $\tau_0 > 0$, and any $r > 2$, where $r$ is the analytic weight parameter appearing in \eqref{eq:norm}. We decompose the initial data  in the following form, written on the Fourier side in the tangential variable 
\begin{align*}
u^{(0)}_\xi(y) = y + \theta_\xi(y, 0) A^{(0)}(\xi) + \bar{w}^{(0)}_\xi(y) \, ,
\end{align*}
where $\theta_\xi$ is defined in \eqref{eq:theta:def}. Assume that  $\bar{w}^{(0)}$ and $A^{(0)}$ satisfy 
\begin{align*}
E_0 := \| (\bar{w}^{(0)}, A^{(0)}) \|_{X_{10 \tau_0}} < \infty,
\end{align*} 
where the analytic energy is defined in \eqref{def:E}. Then there exists a $T_\ast > 0$ depending on $\tau_0, r, E_0$, and there exists a unique solution $(\bar{w}, A)$ to the coupled system \eqref{eq:bar:w:evo}, \eqref{eq:A:evo} with initial datum $(\bar{w}, A)|_{t = 0} = (\bar{w}^{(0)}, A^{(0)})$ such that the total analytic energy $E(t)$ is defined \eqref{def:E} is bounded as
\begin{align*}
\sup_{t \in [0, T_\ast]} E(t) \leq 2 E_0 \,.
\end{align*}
Equivalently, this defines a unique tangentially analytic solution
\begin{align*}
u_\xi = y + \theta_\xi(t,y) A_\xi(t) + \bar{w}_\xi(t,y)
\end{align*}
to the original system \eqref{eq:TD:main}--\eqref{eq:TD:A:p}.
\end{theorem}

The central difficulty towards establishing Theorem~\ref{thm:main} is the apparent loss of two $x$-derivatives in the coupled equations \eqref{eq:TD:main:a} and \eqref{eq:TD:A:p}. Indeed, replacing $-\p_x p$ on the right-hand side of \eqref{eq:TD:main:a} with $\p_x |\p_x|A(x,t)$ according to \eqref{eq:TD:A:p}, and subsequently replacing $A$ with $u(x,y) - y|_{y \rightarrow \infty}$, we see that, in terms of a formal derivative count we have
\begin{align} \label{loss}
\p_t u + u \p_x u + v \p_y u = - \p_x |\p_x| u(x,\infty) \, .
\end{align}  
This loss of \textit{two} $x$ derivatives precludes the  well-posedness of the system even in analytic spaces. 
Our starting point is the observation of skew-adjointedness of the loss term on the right-hand side of \eqref{loss}. Indeed, for any smooth decaying function $g(x)$ one has 
\begin{align}
 \int_\RR g \, \p_x |\p_x| g    = 0 =  \int_\RR |\p_x| g \, \p_x |\p_x| g.
 \label{eq:cancellation}
\end{align}
The cancellation~\eqref{eq:cancellation} holds because we have $|\p_x| = - H \p_x$, where $H$ is the Hilbert transform, and both $H$ and $\p_x$ are skew-adjoint operators on $L^2(\mathbb{R})$. 

This motivates our main reformulation of the system and the extraction of the unknowns we analyze. First, we notice that according to \eqref{eq:TD:BC:c}, $u$ grows like $y$ as $y \to \infty$, while $\p_x u = - \p_y v$ converges to $\p_x A(t,x)$, a bounded function as $y \to \infty$. Hence $v = - I_y[\p_x u]$ also grows like $y$ at $\infty$. Here and throughout the paper we write 
\begin{align}
 I_y[f] = \int_0^y f(y') \ud y' \, .
 \label{eq:funky:notation}
 \end{align}
We are thus led to introduce the expansion of $v$ at $\infty$: 
\begin{align*}
v = y v_1(t,x) + v_0(t,x) + \OO(y^{-1})\quad  \text{ as } \quad y \to \infty.
\end{align*}
The coefficient $v_0(t,x)$ will play a crucial role in the analysis, and is given by the nonlocal integral $I_\infty[\p_x u - \p_x A]$.   We thus reinterpret \eqref{eq:TD:main}  as giving three relations simultaneously, corresponding to the orders of growth as $y \to \infty$. First, collecting the contributions from \eqref{eq:TD:main:a} which are $\OO(y)$ (arising from the terms $u \p_x u$ and $ v \p_y u$), we obtain the asymptotic information that $v_1 = - \p_x A$. Second, we collect the terms which contribute $\OO(1)$ terms at $y = \infty$. This yields a forced Benjamin-Ono equation for the unknown $A(t,x)$: 
\begin{align} \label{A:eq:Intro}
\p_t A + A \p_x A + \p_x |\p_x|A = - v_0 \quad \text{ on } \quad \mathbb{R}.
\end{align}
The cancellation alluded to earlier in \eqref{eq:cancellation} is now readily apparent upon computing the inner-product of $A$ against the Benjamin-Ono equation. 

Having extracted the $\OO(y)$ and $\OO(1)$ contributions, the third step is to extract the  functions in \eqref{eq:TD:main:a} which decay as $y\to \infty$. The relevant unknown, $\bar{w}$, is then a homogenized version of $u$, and it obeys a Prandtl type equation (see~\eqref{eq.hom.w.a} below). This procedure gives rise to the start of our analysis: we analyze simultaneously a Benjamin-Ono equation for $A$, forced by $v_0 = v_0(\bar{w})$, as well as a Prandtl type equation for $\bar{w}$, forced by $A$ related quantities. Summarizing, the simultaneous system of equations we extract are (leaving $F$ an unspecified forcing term for now)
\begin{subequations}
\begin{align} \label{intro:BO}
&\p_t A + A \p_x A + \p_x |\p_x| A = - v_0(\bar{w}) \quad \text{ on } \quad \mathbb{R} \hspace{10 mm} \text{(Benjamin-Ono)}, \\ \label{intro:PR}
&\p_t \bar{w} - \p_y^2 \bar{w} + y \p_x \bar w = F(\bar{w}, A) \hspace{10 mm}  \text{ on } \quad \HH \hspace{10 mm} \text{(Prandtl-type)} \, .
\end{align}
\end{subequations}

The cancellation \eqref{eq:cancellation} applies for the quantity $A$, which describes $u$ at $y = \infty$, and thus \eqref{eq:cancellation} should be interpreted as solving the derivative-loss problem {\em at $y = \infty$}. We now must continue exploiting this cancellation for values of $y < \infty$. Indeed, the two-derivative loss is still lurking for finite $y$ through the forcing term in \eqref{intro:PR}. Specifically,  the reader should consult $\mathcal{B}_\xi(\bar{w}, A)$, defined in \eqref{eq:B:def:*}, and in particular the most singular contributions arise from the $\p_t A_\xi$ term, which in turn create a $i \xi |\xi| A_\xi$ contribution, again yielding a two-derivative loss. Our observation is that such a term is accompanied by a factor of $(1-\theta_\xi)$. By selecting the lift function $\theta_\xi$ in a frequency-dependent manner, we are able to gain back $\langle \p_x \rangle^{3/2}$. The idea of tangential-frequency-dependent boundary layer lifts was also successfully used in~\cite{GVMV18} in the contest of the hydrostatic Navier Stokes equations. For us, the selection of a $\xi$-dependent lift, coupled with Hardy-type inequalities with the homogeneous weights of $y$ enables us to gain back enough regularity {\em near $\{y = 0\}$}.

A further difficulty that arises in our analysis is the loss of one $y$-weight. This occurs due to the non-local integral in \eqref{eq:M:def}, which forces the $\bar{w}$ evolution. In order to handle the loss of a $y$ weight, we control the quantity $y \p_y \bar{w}$ in $L^2$, which is seen in the specification of the $\| \cdot \|_{H_\tau}$ in \eqref{norm:H:tau}. To control this component of the $H_\tau$ norm, we in turn need to commute the vector-field $y \ud y$ with the Prandtl system, which necessitates an analysis of the vorticity equation that governs the evolution of $\p_y \bar{w}$. To successfully analyze the vorticity equation, we capitalize on two essential features. First, we only require this enhanced vector-field for values of $y \ge 1$, so we do not see the boundary effect of the vorticity. Second, we can control the $y \ud y$ in a  {\em weaker norm} in terms of $x$ regularity, which is the reason that the second terms in \eqref{norm:Z:tau} and \eqref{norm:H:tau} are measured on the Sobolev scale $r - 1/2$. This type of {\em lagging norm} structure is essential for our scheme of estimates to close, and in particular prevents a further loss of $y$-weight in the vorticity equation. 

\begin{remark}[Notation] We use heavily the notation $\les$ to suppress universal constants. It is important to emphasize that these universal constants are \textit{independent} of small values of $t, \eps, \delta$, where $\eps$ is the weight parameter in \eqref{eq:rho:def}, and $\delta$ is the parameter appearing in our norms, \eqref{X:tau:norm:def}--\eqref{norm:H:tau}.
\end{remark}

\section{The Prandtl-Benjamin-Ono splitting}
 
\subsection{Benjamin-Ono evolution for $A$}
\label{sec:BO:evo}

We need to understand the asymptotic behavior at $y = \infty$ a bit more carefully. First, from \eqref{eq:TD:BC} we obtain
\begin{align*}
u \sim y + A(x,t), \qquad \p_x u \rightarrow \p_x A, \qquad \p_t u \rightarrow \p_t A, \qquad \p_y u \rightarrow1 ,  \, \qquad \text{ as } y \rightarrow \infty.
\end{align*}
The function $v(x,y,t) = - I_y[\p_x u](x,t) = - \int_0^y \partial_x u(x,z,t) dz$ is expected to grow like $y$ at $\infty$, so we let 
\begin{align}
v \sim v_0(x,t) + v_1(x,t)y \, \qquad \text{ as } y \rightarrow \infty.  
\label{eq:v:asymptiotic:1}
\end{align}
We now evaluate the original equation \eqref{eq:TD:main:a} at $y = \infty$ and use the above information to obtain
\begin{align} \label{eq:A:w:y}
\p_t A + (y + A) \p_x A + (v_0 + v_1 y) + \p_x |\p_x| A = 0.   
\end{align}
Due to the super-exponential weights in $y$, $\rho(t,y)$, appearing in our norm  \eqref{eq:norm}, we guarantee that the remaining terms in \eqref{eq:TD:main:a} vanish sufficiently rapidly as $y \to \infty$ so as to not contribute towards \eqref{eq:A:w:y}. From here, we extract two equations by matching the orders of $y$ for $y \to \infty$: 
\begin{subequations}
\begin{align} 
&\p_x A + v_1 = 0,  \label{eq.v1} \\ 
&\p_t A(x,t) + A\p_x A + v_0 + \p_x |\p_x|A = 0. \label{eq.v0}
\end{align}
\end{subequations}
We now compute the function  $v_0$ in a different fashion: 
\begin{align}
v =  - I_y[\p_x u] = - I_y  \Big[ (\p_x u - \p_x A) + \p_x A \Big] =  - y \p_x A -I_y [\p_x u - \p_x A].
\label{eq:v:asymptiotic:2}
\end{align}
Here we use the notation in \eqref{eq:funky:notation} for $I_y[\cdot]$.  
From \eqref{eq:v:asymptiotic:1} and \eqref{eq:v:asymptiotic:2} we deduce that $v_1 = - A_x$ and  that
\begin{align}
v_0(x,t) = -I_\infty [\p_x u(x,y,t) - \p_x A(x,t)]  \, .
\label{eq:v0:def}
\end{align} 
Thus, $v_0$ can be expressed in terms of $u$ and $A$. To emphasize this, we will write $v_0 = v_0(u,A)$. Note that we need to understand $u$ (or $\p_x u$) for all $y$ in order to understand $v_0$ (it is nonlocal).  Inserting back into \eqref{eq.v0}, we obtain the evolution equation for $A$:
\begin{align} \label{eq.A}
\p_t A + A \p_x A  + \p_x |\p_x| A = - v_0(u,A) \text{ for } x \in \mathbb{R}. 
\end{align}

\subsection{Prandtl-type evolution for $\bar{w}$}
\label{sec:Prandtl:evo}

The first step towards homogenizing the boundary conditions for $u$ in the equation \eqref{eq:TD:main:a} is to remove the linear profile $y$ and introduce the unknown 
\begin{align*}
 w = u - y 
\end{align*}
so that \eqref{eq:TD:BC:a}--\eqref{eq:TD:BC:c} yield
\begin{align*}
w|_{y = 0} = w|_{x = -\infty} = w|_{x = \infty} = 0, \qquad w|_{y \to \infty} = A(x,t). 
\end{align*}
We do not need to change $v$ here, as it is given by $-I_y(\partial_x u) = -I_y( \partial_x w)$. It follows that the evolution equation for $w$ is
\begin{align*}
\p_t w + w \p_x w + (y \p_x w +v) + v \p_y w  - \p_{y}^2 w  + \p_x |\p_x|A(x,t) =0. 
\end{align*}
Summarizing, the unknowns $w$ and $v$ take the place of the usual Prandtl unknowns, and the equation obeyed by $w$ is nothing but the usual Prandtl system with a  few extra linear terms
\begin{subequations}
\label{eq:w:evolution}
\begin{align} 
&\p_t w - \p_{yy}w + w\p_x w + v\p_y w + (y \p_x w + v)  + \p_x |\p_x| A = 0, \label{start.Prandtl.a} \\
&w|_{y = 0} = w|_{x = - \infty} = w|_{x = +\infty} = 0, \qquad w|_{y = \infty} = A(x,t),  \label{start.Prandtl.b} \\ 
&\p_x w +\p_yv = 0, \quad  v|_{y = 0} = 0, \quad  \Rightarrow v = - I_y[ \p_x w],  \label{start.Prandtl.c} 
\end{align}
\end{subequations}
The system \eqref{eq:w:evolution} is of course coupled to the evolution equation for $A$ given in \eqref{eq.A}.

In order to analyze the system \eqref{eq:w:evolution}, it is convenient to homogenize the boundary condition of $w$ as $y\to \infty$. For this purpose we introduce a tangential frequency dependent lift of the normal boundary condition, so that we need to write the system obeyed by the Fourier transform in the $x$ variable of \eqref{start.Prandtl.a}--\eqref{start.Prandtl.c}. This yields 
\begin{subequations}
\begin{align*}
&\p_t w_\xi - \p_{yy} w_\xi + (w \partial_x w + v \partial_y w)_\xi + i \xi (y w_\xi - I_y[w_\xi]) + i \xi |\xi| A_\xi  = 0, \\
&w_\xi|_{y = 0} = 0, \qquad w_\xi|_{y \rightarrow \infty} = A_\xi, \\
&i \xi w_\xi + \p_y v_\xi = 0, \qquad v_\xi|_{y= 0} = 0 \Rightarrow v_\xi = - i \xi I_y[w_\xi]. 
\end{align*}
\end{subequations}
For each $\xi \in \RR$ we introduce a lift function $\theta_\xi$ (given explicitly in \eqref{eq:theta:def} below) and define new unknowns, 
\begin{align*}
\bar{w}_\xi &:= w_\xi(y,t) - A_\xi(t) \theta_\xi(y,t) \\ 
\bar{v}_\xi &:= v_\xi + i \xi A_\xi I_y[\theta_\xi]
\,.
\end{align*}
We derive from \eqref{eq:w:evolution} the evolution for $\bar w_\xi$, which reads  
\begin{subequations}
\label{eq:bar:w:evo}
\begin{align} \label{eq.hom.w.a}
&\p_t \bar{w}_\xi - \p_{yy}\bar{w}_\xi + i \xi  y \bar{w}_\xi + \Ncal_\xi(\bar{w}, \bar{w})   + \Lcal_\xi (\bar{w} , A ) + \Mcal_{\xi} (\bar{w} , A ) + \Bcal_\xi (\bar{w}, A) = 0, \\  \label{eq.hom.w.a.b}
&\bar{v}_\xi = - i\xi I_y [\bar{w}_\xi] , \\ \label{eq.hom.w.a.c}
&\bar{w_\xi}|_{y=0} = \bar{w_\xi}|_{y\to \infty} = 0, 
\end{align}
\end{subequations}
where in \eqref{eq.hom.w.a} above we have defined  
\begin{subequations}
\label{eq:NLMB:def}
\begin{align} 
\Ncal_\xi ( \bar{w}, \bar{w} ) 
&:= i \int_{\RR} \Big(\bar w_\eta (\xi-\eta) \bar w_{\xi-\eta}   - \eta I_y[\bar w_\eta] \partial_y \bar w_{\xi-\eta}  \Big)  \ud\eta 
\label{eq:N:def}
\\ 
\Lcal_\xi (\bar{w} , A ) 
&:= i \int_{\RR}\biggl( \bar w_\eta (\xi-\eta) A_{\xi-\eta} \theta_{\xi-\eta} + A_\eta \theta_\eta (\xi-\eta) \bar w_{\xi-\eta}    - \eta I_y[\bar w_\eta] A_{\xi-\eta} \partial_y \theta_{\xi-\eta} \biggr) \ud\eta 
\label{eq:L:def} 
\\ 
\Mcal_{\xi}(\bar{w}, A)
&:= -i \int_{\RR} \eta A_\eta I_y[\theta_\eta] \p_y \bar w_{\xi-\eta}  \ud \eta
\label{eq:M:def}
\\ 
\Bcal_\xi (\bar{w},A) 
&:= A_\xi (\partial_t - \partial_{yy})\theta_\xi + \left(\theta_\xi -1\right) \partial_t A_\xi + (\p_t A_\xi + i \xi |\xi| A_\xi)  + i\xi  \Big( A_\xi \left( y \theta_\xi  - I_y[\theta_\xi] \right) -I_y[\bar w_\xi]  \Big)  \notag\\
&\qquad  
  + i \int_{\RR} \Big(  A_\eta \theta_\eta (\xi-\eta) A_{\xi-\eta} \theta_{\xi-\eta} - \eta A_\eta I_y[\theta_\eta] A_{\xi-\eta} \partial_y \theta_{\xi-\eta} \Big) \ud \eta
\,.\label{eq:B:def}
\end{align}
\end{subequations}

At this stage, we make the following choice for the lift function 
\begin{align}
\theta_\xi(y,t)  = 1 - e^{- \frac{y^2 \brak{\xi}^2}{2(1+t/\eps)}} \, ,
\label{eq:theta:def}
\end{align}
where $\eps>0$ is a parameter to be chosen later. We emphasize here that $\theta_\xi(y,0)$ {\em does not depend on $\eps$}, which is crucial for the proof. Informally, $\eps$ will be selected small relative to universal constants, and relative to the size of the initial data (which is independent of $\eps$). The time of existence $T_\ast$ will be selected small relative to $\eps$ and in particular we restrict ourselves to $T_\ast \le \eps$, so that the quotient $t/\eps$ is always bounded by $1$. 

It is also useful to denote 
\begin{align} \label{c:theta:xi}
c_{\theta,\xi}(t) := I_\infty[1-\theta_\xi](t) = \int_0^\infty (1-\theta_\xi(y,t)) dy = \int_0^\infty e^{- \frac{y^2 \brak{\xi}^2}{1+t/\eps}} dy  = \frac{ \sqrt{\pi (1+t/\eps)}}{2 \brak{\xi}} \,. 
\end{align}
With $\theta_\xi$ as defined by \eqref{eq:theta:def}, we  identify the function $ v_0(\bar w, A)$ from \eqref{eq:v0:def} as 
\begin{align}
 (v_0(\bar w, A))_\xi(t) &= - i\xi \int_0^\infty \left(   \bar w_\xi(y,t) -  A_\xi(t) (1-\theta_\xi(y,t)) \right) \ud y \notag\\
 &= - i\xi I_{\infty}[\bar w_\xi] (x,t) + i\xi c_{\theta,\xi}(t) A_\xi(t)  
 \,. 
 \label{eq:v0:new}
\end{align}
With this notation, we return to  \eqref{eq.A} which in view of \eqref{eq:v0:new} becomes
\begin{align} 
\p_t A_\xi + i\xi c_{\theta,\xi} A_\xi  - i\xi I_{\infty}[\bar w_\xi] + i \xi |\xi| A_\xi = - i \int_{\RR} A_\eta (\xi-\eta) A_{\xi-\eta} \ud \eta \, .
\label{eq:A:evo}
\end{align}
We notice that $A$ enters the evolution equation for $\bar w$ only through the coefficients of $\Bcal, \Lcal$, and $\Mcal$, whereas $\partial_x \bar w$ enters the evolution equation for $A$ only thought its vertical mean encoded in $v_0(\bar w,A)$. 

Lastly, using that $A$ obeys the Benjamin-Ono equation \eqref{eq:A:evo}, and using that $c_{\theta,\xi} = I_\infty[1-\theta_\xi]$, we may rewrite the forcing term $\Bcal_\xi(\bar w,A)$ given in \eqref{eq:B:def} as  
\begin{align}
\Bcal_\xi (\bar{w},A) 
& = A_\xi (\partial_t - \partial_{yy})\theta_\xi + \left(\theta_\xi -1\right) \partial_t A_\xi   + i\xi \left( I_\infty[\bar w_\xi] - I_y[\bar w_\xi] \right) \notag\\
&\qquad + i\xi   A_\xi \bigl( y (\theta_\xi-1)  - ( I_\infty[1-\theta_\xi] - I_y[1-\theta_\xi] )\bigr)   \notag\\
&\qquad + i \int_{\RR} \bigl(   (\xi-\eta) A_\eta  A_{\xi-\eta} \left( \theta_\eta \theta_{\xi-\eta} -1\right) - \eta A_\eta A_{\xi-\eta} I_y[\theta_\eta] \partial_y \theta_{\xi-\eta} \bigr)
  \label{eq:B:def:*}
\end{align}
Because of our choice of $\theta_\xi$, every single term in  $\Bcal_\xi(\bar v,A)$ decays to $0$ as $y\to \infty$. 
Throughout the rest of the paper, we use the formulation \eqref{eq:B:def:*} of the $\Bcal_\xi$ term (instead of \eqref{eq:B:def}).

\section{Energy estimates and the proof of the Main Theorem}
\label{sec:energy}

In this section we give the energy estimates which prove Theorem~\ref{thm:main}, under the assumption that the nonlinear terms may be bounded suitably (cf.~Lemma~\ref{lem:main}). These terms are then estimated in Section~\ref{sec:main:lemma}.

\subsection{Energy inequality for $A$}
In view of \eqref{eq:ODE:tilde:XY} we  take product of equation \eqref{eq:A:evo} with the complex conjugate $\overline{{A}_\xi}$, and integrate in $\xi$ against $e^{2\tau |\xi|} \langle \xi \rangle^{2r} \langle \xi \rangle$ to obtain
\begin{align}
\brak{\partial_t A, A}_{\tilde{\tau,r}} 
&= - \int_{\RR} e^{2\tau|\xi|} \brak{\xi}^{2r} \left(i \xi (c_{\theta, \xi} +|\xi|) |A_\xi|^2   - i \xi I_{\infty}[\bar{w}_\xi] \overline{A_\xi} + i \overline{A_\xi} \int_{\mathbb{R}} A_\eta(\xi - \eta) A_{\xi - \eta} \ud \eta \right) \ud \xi
\notag\\
&= T_{\mathcal A,1} - T_{\mathcal A,2}
\label{eq:t:A:energy:*}
\end{align}
where we have defined
\begin{subequations}
\label{eq:T:A12}
\begin{align}
T_{\mathcal A,1} &= \brak{I_{\infty}[\p_x \bar w], A}_{\tilde{\tau,r}} =  \int_{\RR}  i \xi I_\infty[\bar{w}_\xi] \overline{A_\xi} e^{2 \tau |\xi|} \langle \xi \rangle^{2r}  \ud \xi 
\label{eq:T:A12:a}\\
T_{\mathcal A,2} &= \brak{A \partial_x A, A}_{\tilde{\tau,r}} =  \int_{\RR} \int_{\RR} i (\xi - \eta)   A_\eta A_{\xi - \eta}  \overline{A_\xi} e^{2 \tau |\xi|} \langle \xi \rangle^{2r} \ud \eta \ud \xi
\label{eq:T:A12:b}
\, .
\end{align}
\end{subequations}
In \eqref{eq:dt:A:energy} we have used that $A$ is real-valued, so that $A_{-\xi} = \overline{A_{\xi}}$, and that $c_{\theta, \xi} = c_{\theta,-\xi} \in \mathbb{R}$ (cf.~\eqref{c:theta:xi}). Combining \eqref{eq:ODE:tilde:XY} with \eqref{eq:t:A:energy:*} we arrive  at
\begin{align}
 \frac{d}{2 dt} \norm{A}_{\tilde{\tau,r}}^2  +  (-\dot{\tau}) \norm{|\partial_x|^{1/2} A}_{\tilde{\tau,r}}^2 & \leq \abs{T_{\mathcal A,1}} +\abs{T_{\mathcal A,2}}
\label{eq:dt:A:energy}
\end{align}
which is the desired energy inequality for the analytic norm of $A$. The terms on the right side of \eqref{eq:dt:A:energy} are estimated in Lemma~\ref{lem:main}, bounds \eqref{eq:Shaq:A}.

\subsection{Energy inequality for $\bar{w}$}
In view of \eqref{eq:ODE:XY} we need to compute $\brak{\partial_t \bar w + \bar w (\partial_t \log) \rho,\bar w}_{\tau,r}$. Note that by the definition \eqref{eq:rho:def} we have 
\begin{align}
\partial_t (\log \rho) = -\frac{y^2}{8 \eps (1+t/\eps)^2} 
\label{eq:LeBron:0}
\end{align}
and thus we obtain the damping weight-gaining term
\begin{align}
\brak{\bar w \partial_t (\log \rho),\bar w}_{\tau,r} = -\frac{1}{8 \eps (1+t/\eps)^2} \norm{y \bar w}_{\tau,r}^2
\, .
\label{eq:LeBron:1}
\end{align}
In order to compute $\brak{\partial_t \bar w,\bar w}_{\tau,r}$, we multiply  \eqref{eq.hom.w.a} with $\rho^2 \overline{\bar w_\xi} e^{2\tau|\xi|} \brak{\xi}^{2r}$ and integrate over $(\xi,y) \in\RR\times [0,\infty)$ to obtain
\begin{align}
\brak{\partial_t \bar w,\bar w}_{\tau,r}
&= - \norm{\partial_y \bar w}_{\tau,r}^2 - \frac{1}{2(1+t/\eps)} \int_{\RR} \int_0^\infty \partial_y \bar w_\xi y \overline{\bar w_\xi} \rho^2 e^{2\tau|\xi|} \brak{\xi}^{2r} \ud \eta \ud \xi \notag\\
&\qquad - i \int_{\RR} \int_0^\infty \xi \abs{\bar w_\xi}^2  \rho^2 e^{2\tau|\xi|} \brak{\xi}^{2r} \ud \eta \ud \xi - T_{\Ncal} - T_{\Lcal} - T_{\Mcal} - T_{\Bcal} \notag\\
&\leq -\frac 12 \norm{\partial_y \bar w}_{\tau,r}^2 +\frac{1}{8(1+t/\eps)^2} \norm{y \bar w}_{\tau,r}^2 + \abs{T_{\Ncal}} + \abs{T_{\Lcal}} + \abs{T_{\Mcal}} + \abs{T_{\Bcal}} \,
\label{eq:LeBron:2}
\end{align}
where we have used  that by oddness in $\xi$ we have
\begin{align*}
i \int_{\RR} \int_0^\infty \xi \abs{\bar w_\xi}^2  \rho^2 e^{2\tau|\xi|} \brak{\xi}^{2r} \ud \eta \ud \xi
= 0
\end{align*}
and we have denoted
\begin{subequations}
\label{eq:T:cal:1}
\begin{align}
T_{\Ncal} &= \int_{\RR} \int_0^\infty \Ncal_\xi (\bar w, \bar w) \overline{\bar w_\xi} \rho^2 e^{2\tau|\xi|} \brak{\xi}^{2r} \ud y \ud \xi
\\
T_{\Lcal} &= \int_{\RR} \int_0^\infty \Lcal_\xi(\bar w,A)  \overline{\bar w_\xi} \rho^2 e^{2\tau|\xi|} \brak{\xi}^{2r} \ud y \ud \xi
\\
T_{\Mcal} &= \int_{\RR} \int_0^\infty \Mcal_\xi(\bar w,A)  \overline{\bar w_\xi} \rho^2 e^{2\tau|\xi|} \brak{\xi}^{2r} \ud y \ud \xi
\\
T_{\Bcal} &= \int_{\RR} \int_0^\infty \Bcal_\xi(\bar w,A)  \overline{\bar w_\xi} \rho^2 e^{2\tau|\xi|} \brak{\xi}^{2r} \ud y \ud \xi
\end{align}
\end{subequations}
with $\Ncal_\xi, \Lcal_\xi, \Mcal_\xi, \Bcal_\xi$, as defined in \eqref{eq:NLMB:def}.
Combining  \eqref{eq:ODE:XY} with \eqref{eq:LeBron:1}--\eqref{eq:LeBron:2} we arrive at  
\begin{align}
\frac{d}{2 dt} \norm{\bar w}_{\tau,r}^2 +  (-\dot{\tau}) \norm{|\partial_x|^{1/2} \bar w}_{\tau,r}^2 + \frac 12 \norm{\partial_y \bar w}_{\tau,r}^2  + \frac{1-\eps}{8 \eps (1+t/\eps)^2} \norm{y \bar w}_{\tau,r}^2
\leq  \abs{T_{\Ncal}} +  \abs{T_{\Lcal}} +  \abs{T_{\Mcal}}  +  \abs{T_{\Bcal}}  
\label{eq:LeBron:3}
\end{align}
which is the desired energy inequality for the analytic norm of $\bar w$. The  four error terms on the right side of \eqref{eq:LeBron:3} are estimated in Lemma~\ref{lem:main}, bounds~\eqref{eq:Shaq:w}.

\subsection{Energy inequality for $\p_y \bar{w}$}
In order to overcome a loss of $y$ weight in the term $T_{\Mcal}$, we need to also consider the evolution of the normalized vorticity $\p_y \bar w$. We apply $\p_y$ to \eqref{eq.hom.w.a} to obtain 
\begin{align} \label{eq:diff:y}
\p_t \p_y \bar{w}_\xi - \p_y^2 \p_y \bar{w}_\xi + i \xi \bar{w}_\xi + i y \xi \p_y \bar{w}_\xi 
= - \left(\p_y  \mathcal{N}_\xi +  \p_y \mathcal{L}_\xi + \p_y  \mathcal{M}_{\xi} +  \p_y \mathcal{B}_\xi \right). 
\end{align}
Note that some of the terms on the right side of \eqref{eq:diff:y} have cancelations in them.
Using that $I_y[\partial_y \bar{w}_\xi] = \bar{w}_\xi = \partial_y I_y[\bar w_\xi]$, and upon noting symmetries $\eta \leftrightarrow \xi-\eta$ in the below integrals,
we may rewrite
\begin{subequations}
\label{eq:dy:nonlinear:terms}
\begin{align}
\partial_y \Ncal_\xi(\bar w,\bar w) 
&= \Ncal_\xi(\bar w, \p_y \bar w) = i \int_{\RR} \bigr(\bar w_\eta (\xi-\eta) \p_y \bar w_{\xi-\eta}   - \eta I_y[\bar w_\eta] \partial_{yy} \bar w_{\xi-\eta}  \bigl)  \ud\eta  
\label{eq:dy:nonlinear:terms:a}
\\
\partial_y \Lcal_\xi(\bar w,\bar w) 
&=  i \int_{\RR}\bigl(\p_y \bar w_\eta (\xi-\eta) A_{\xi-\eta} \theta_{\xi-\eta} + \bar w_{\eta}  (\xi-\eta)A_{\xi-\eta}\p_y \theta_{\xi-\eta}   \notag\\
&\qquad \qquad \qquad + A_\eta \theta_\eta (\xi-\eta) \p_y \bar w_{\xi-\eta} - \eta I_y[\bar w_\eta] A_{\xi-\eta} \partial_{yy} \theta_{\xi-\eta} \bigr) \ud\eta 
\label{eq:dy:nonlinear:terms:b}
\\
\partial_y  \Mcal_\xi(\bar w,A) 
&= -i \int_{\RR} \bigl( \eta A_\eta  \theta_\eta  \p_y \bar w_{\xi-\eta} + \eta A_\eta I_y[\theta_\eta] \p_{yy} \bar w_{\xi-\eta} \bigr) \ud \eta 
\label{eq:dy:nonlinear:terms:c}
\\
\partial_y  \Bcal_\xi(\bar w,A) 
&= A_\xi (\partial_t - \partial_{yy}) \partial_y \theta_\xi + \partial_y \theta_\xi \partial_t A_\xi- i\xi \bar w_\xi + i \xi A_\xi y \partial_y \theta_\xi
\notag\\
&\qquad \qquad \qquad + i \int_{\RR} \bigl(   (\xi-\eta) A_\eta  A_{\xi-\eta}  \theta_\eta \p_y \theta_{\xi-\eta}   - \eta A_\eta A_{\xi-\eta} I_y[\theta_\eta] \partial_{yy} \theta_{\xi-\eta} \bigr)
\label{eq:dy:nonlinear:terms:d}
\, .
\end{align}
\end{subequations}
It turns out that we only need information on the vorticity $\partial_y \bar w$ away from $\partial \HH = \{y= 0\}$, and for this purpose we introduce a cut-off function, $\chi = \chi(y)$, such that $0 \leq \chi' \leq 1$, satisfying 
\begin{align}
\chi(y) = 
\begin{cases} 
0, \quad  \text{ on } y \in [0, 1) \\ 
1, \quad \text{ on } y \ge 6  \, .
\end{cases}
\label{eq:chi:def}
\end{align}
Note that $\chi$ is independent of time.
Our goal is to estimate $\norm{\chi \, \partial_y \bar w}_{\tau,r-1/2}$. The shift in Sobolev regularity of for the vorticity norm, i.e. the change $r\mapsto r-1/2$, is essential for the energy estimate to close.

Using \eqref{eq:ODE:XY} with $f = \chi \, \partial_y \bar w$, property \eqref{eq:LeBron:0} of the weight $\rho$, and the evolution equation \eqref{eq:diff:y} we obtain
\begin{align}
&
\frac{d}{2dt}  \norm{\chi   \p_y \bar w}_{\tau,r-1/2}^2 + (-\dot \tau) \norm{\chi |\partial_x|^{1/2} \p_y \bar w}_{\tau,r-1/2}^2
\notag\\
&= \brak{\chi \, \p_t \p_y \bar w, \chi \, \p_y \bar w}_{\tau,r-1/2} + \brak{\chi \, \p_y \bar w \p_t (\log \rho), \chi \, \p_y \bar w}_{\tau,r-1/2}
\notag\\
&= -\frac{1}{8\eps (1+t/\eps)^2} \norm{y \chi \, \p_y \bar w}_{\tau,r-1/2}^2 
- \norm{\chi \p_{yy} \bar w}_{\tau,r-1/2}^2 - \int_{\RR} \int_0^\infty \p_{yy} \bar w_\xi \p_y(\chi^2 \rho^2) \overline{\p_y \bar w_\xi} \brak{\xi}^{2r-1} e^{2\tau|\xi|} \ud y \ud \xi
\notag\\
&\qquad  - i \int_{\RR} \int_0^\infty \xi \bar w_\xi  \chi^2 \rho^2  \overline{\p_y \bar w_\xi} \brak{\xi}^{2r-1} e^{2\tau|\xi|} \ud y \ud \xi - T_{\p_y \Ncal} - T_{\p_y \Lcal} - T_{\p_y \Mcal} - T_{\p_y \Bcal}  \,.
\label{eq:Kobe:0}
\end{align}
Here we have used the cancellation property
\begin{align*}
i \int_{\RR} \int_0^\infty y \xi \abs{\p_y \bar w_\xi}^2 \chi^2 \rho^2 \brak{\xi}^{2r-1} e^{2\tau |\xi|} \ud y \ud \xi = 0
\end{align*}
and have denoted
\begin{subequations}
\label{eq:Kobe:1}
\begin{align}
T_{\p_y \Ncal} &= \int_{\RR} \int_0^\infty \p_y \Ncal_\xi (\bar w, \bar w) \overline{\p_y \bar w_\xi} \chi^2 \rho^2 e^{2\tau|\xi|} \brak{\xi}^{2r-1} \ud y \ud \xi
\label{eq:Kobe:1:a}
\\
T_{\p_y \Lcal} &= \int_{\RR} \int_0^\infty \p_y \Lcal_\xi(\bar w,A)  \overline{\p_y \bar w_\xi} \chi^2  \rho^2 e^{2\tau|\xi|} \brak{\xi}^{2r-1} \ud y \ud \xi
\label{eq:Kobe:1:b}
\\
T_{\p_y \Mcal} &= \int_{\RR} \int_0^\infty \p_y \Mcal_\xi(\bar w,A)  \overline{\p_y \bar w_\xi} \chi^2  \rho^2 e^{2\tau|\xi|} \brak{\xi}^{2r-1} \ud y \ud \xi
\label{eq:Kobe:1:c}
\\
T_{\p_y \Bcal} &= \int_{\RR} \int_0^\infty \p_y \Bcal_\xi(\bar w,A)  \overline{\p_y \bar w_\xi} \chi^2  \rho^2 e^{2\tau|\xi|} \brak{\xi}^{2r-1} \ud y \ud \xi
\label{eq:Kobe:1:d}
\end{align}
\end{subequations}
with $\p_y \Ncal_\xi, \p_y \Lcal_\xi, \p_y \Mcal_\xi, \p_y \Bcal_\xi$, as defined in \eqref{eq:dy:nonlinear:terms}.
From \eqref{eq:Kobe:0}, the Cauchy-Schwarz inequality and the definitions of $\chi$ and $\rho$ we obtain
\begin{align}
&
\frac{d}{2dt}  \norm{\chi   \p_y \bar w}_{\tau,r-1/2}^2 + (-\dot \tau) \norm{\chi |\partial_x|^{1/2} \p_y \bar w}_{\tau,r-1/2}^2 + \frac{1-\eps}{8\eps (1+t/\eps)^2} \norm{y \chi \, \p_y \bar w}_{\tau,r-1/2}^2  + \frac 14\norm{\chi \p_{yy} \bar w}_{\tau,r-1/2}^2
\notag\\
&\qquad \leq  \norm{\p_y \bar w}_{\tau,r}^2 +  \norm{ \bar w}_{\tau,r} \norm{ \p_y \bar w}_{\tau,r}
+ \abs{T_{\p_y \Ncal}} +\abs{T_{\p_y \Lcal}} +\abs{T_{\p_y \Mcal}} +\abs{T_{\p_y \Bcal}}  \,.
\label{eq:Kobe:2}
\end{align}
Here we have used that $\brak{\xi}\geq 1$. The remaining four error terms on the right side of \eqref{eq:Kobe:2} are bounded in Lemma~\ref{lem:main}, estimate~\eqref{eq:Shaq:dy:w}.

\subsection{Nonlinear estimates}
The following lemma summarizes the available estimates for the error terms in \eqref{eq:dt:A:energy}, \eqref{eq:LeBron:3}, and \eqref{eq:Kobe:2}.
\begin{lemma}[\bf Main Nonlinear Lemma] \label{lem:main}
Assume that $r>2$ and that $t \le \eps$. For the error terms in the $A$ energy estimate~\eqref{eq:dt:A:energy} we have
\begin{subequations}
\label{eq:Shaq:A}
\begin{align}
\abs{T_{\mathcal A,1}} &\les  \norm{|\p_x|^{1/2} \bar w}_{\tau,r}  \norm{|\p_x|^{1/2} A}_{\tilde{\tau,r}}
\\
\abs{T_{\mathcal A,2}} &\les \norm{|\p_x|^{1/2} A}_{\tilde{\tau,r}}^2 \norm{A}_{\tilde{\tau,r}}  + \norm{|\p_x|^{1/2} A}_{\tilde{\tau,r}} \norm{A}_{\tilde{\tau,r}}^2
\end{align}
\end{subequations}
for the error terms in the $\bar w$ energy estimate~\eqref{eq:LeBron:3} it holds that
\begin{subequations}
\label{eq:Shaq:w}
\begin{align}
\abs{T_{\mathcal N}} &\les \norm{|\p_x|^{1/2} \bar w}_{\tau,r} \norm{\brak{\p_x}^{1/2} \bar w}_{\tau,r}    \norm{\p_y \bar w}_{\tau,r}  
\label{eq:Barkley:*}
\\
\abs{T_{\mathcal L}} &\les 
\norm{|\p_x|^{1/2} \bar w}_{\tau,r} \norm{\bar w}_{\tau,r} \norm{|\p_x|^{1/2} A}_{\tilde{\tau,r}} + 
\left(\norm{\bar w}_{\tau,r}^2 + \norm{|\p_x|^{1/2} \bar w}_{\tau,r}^2\right) \norm{A}_{\tilde{\tau,r}} 
\label{eq:Jordan:*}
\\
\abs{T_{\mathcal M}} &\les 
 \norm{\brak{\p_x}^{1/2} \bar w}_{\tau,r}  \norm{\brak{\p_x}^{1/2} A}_{\tilde{\tau,r}}  \left( \norm{\p_y \bar w}_{\tau,r} +  \norm{y \chi \p_y \bar w}_{\tau,r-1/2}   \right)
\label{eq:Jordan:**}
\\
\abs{T_{\mathcal B}} &\les 
\frac{1}{\eps} \norm{\bar w}_{\tau,r} \norm{A}_{\tilde{\tau,r}} + 
  \norm{\brak{\p_x}^{1/2} A}_{\tau,r} \norm{\p_y \bar w}_{\tau,r}  +  \norm{\abs{\p_x}^{1/2} \bar w}_{\tau,r}^2
\notag\\
&\quad +
 \norm{|\p_x|^{1/2} \bar w}_{\tau,r} \norm{\bar w}_{\tau,r}  +
  \norm{\brak{\p_x}^{1/2} \bar w}_{\tau,r}  \norm{|\p_x|^{1/2} A}_{\tilde{\tau,r}} \norm{A}_{\tilde{\tau,r}}
\label{eq:Magic:*}
\end{align}
\end{subequations}
while for the error terms in the energy estimate~\eqref{eq:Kobe:2} for $\p_y \bar w$ the estimates
\begin{subequations}
\label{eq:Shaq:dy:w}
\begin{align} 
\abs{T_{\p_y \mathcal N}} &\les 
 \norm{\brak{\p_x}^{1/2} \chi \p_y \bar w}_{\tau,r-1/2}^2 \norm{\p_y \bar w}_{\tau,r} +  \norm{\chi \p_{yy} \bar w}_{\tau,r-1/2}  \norm{\chi \p_{y} \bar w}_{\tau,r-1/2}   \norm{\brak{\p_x}^{1/2} \bar w}_{\tau,r}
\label{eq:Barkley:@}
\\
\abs{T_{\p_y  \mathcal L}} &\les 
\norm{\chi \p_y \bar w}_{\tau,r-1/2}^2 \norm{|\p_x|^{1/2} A}_{\tilde{\tau,r}}
+  \norm{\chi \p_y \bar w}_{\tau,r-1/2} \norm{\brak{\p_x}^{1/2} \bar w}_{\tau,r} \norm{A}_{\tilde{\tau,r}}
\notag\\
&\quad + \norm{\chi |\p_x|^{1/2} \p_y \bar w}_{\tau,r-1/2}  \norm{\chi \brak{\p_x}^{1/2} \p_y \bar w}_{\tau,r-1/2} \norm{A}_{\tilde{\tau,r}}
\label{eq:Jordan:@}
\\
\abs{T_{\p_y  \mathcal M}} &\les \norm{\chi \p_y \bar w}_{\tau,r-1/2}^2 \norm{\abs{\p_x}^{1/2} A}_{\tilde{\tau,r}} + \norm{y \chi \p_y \bar w}_{\tau,r-1/2} \norm{\chi \p_{yy} \bar w}_{\tau,r-1/2} \norm{\abs{\p_x}^{1/2} A}_{\tilde{\tau,r}}
\label{eq:Jordan:@:*}
\\ \n
\abs{T_{\p_y \mathcal B}} &\les 
\frac{1}{\eps}   \norm{A}_{\tilde{\tau,r}} \norm{\chi \p_y\bar w}_{\tau,r-1/2}
+  \norm{\brak{\p_x}^{1/2} \bar w}_{\tau,r} \norm{\chi \p_y\bar w}_{\tau,r-1/2}  +  \norm{\brak{\p_x}^{1/2} A}_{\tilde{\tau,r}} \norm{\chi \p_y \bar w}_{\tau,r-1/2} 
\\  &+  \norm{A}_{\tilde{\tau,r}}^2 \norm{\chi \p_y \bar w}_{\tau,r-1/2} + \norm{\abs{\p_x}^{1/2} A}_{\tilde{\tau,r}}  \norm{A}_{\tilde{\tau,r}}  \norm{\chi \p_y \bar w}_{\tau,r-1/2}
\label{eq:Magic:@}
\end{align}
\end{subequations}
hold. The implicit constants in the above estimates are independent of $\tau$, $t$, and $\eps$ (they depend solely on $r$). 
\end{lemma}
The proof of Lemma~\ref{lem:main} is given in Section~\ref{sec:main:lemma} below. Assuming that this lemma is established, we continue with the proof of the main theorem. Before doing so, we summarize the bounds proven in Lemma~\ref{lem:main} using the total norms defined in \eqref{eq:cluster:fuck} above. Estimate \eqref{eq:Shaq:A} shows that
\begin{align}
\left( \abs{T_{\mathcal A,1}} +  \abs{T_{\mathcal A,2}} \right) \les  \left( 1 +   \norm{(\bar w,A)}_{X_\tau}\right) \norm{(\bar w,A)}_{Y_\tau}^2 +   \norm{(\bar w,A)}_{X_\tau}^3.
\label{eq:A:error:total}
\end{align}
The bounds \eqref{eq:Shaq:w} and an $\eps$-Young inequality for the second term in \eqref{eq:Magic:*} yields 
\begin{align}
&\abs{T_{\Ncal}}+\abs{T_{\Lcal}}+\abs{T_{\Mcal}}+\abs{T_{\Bcal}}\notag\\
&\quad \les \norm{(\bar w,A)}_{Y_\tau}^2  
  \left( \eps^{-1}  + \norm{\bar w}_{Z_\tau} +  \norm{(\bar w,A)}_{X_\tau} + \delta   \norm{\bar w}_{H_\tau}\right) 
\notag\\
&\qquad + \norm{(\bar w,A)}_{X_\tau}^2  \left(\eps^{-1} +  \norm{\bar w}_{Z_\tau} + \norm{(\bar w,A)}_{X_\tau} + \delta    \norm{\bar w}_{H_\tau} \right)
  + \eps \norm{\bar w}_{Z_\tau}^2 
\label{eq:w:error:total}
\end{align}
while the inequality \eqref{eq:Shaq:dy:w} implies 
\begin{align}
&\delta^{-2}\left(\abs{T_{\p_y \Ncal}}+\abs{T_{\p_y \Lcal}}+\abs{T_{\p_y \Mcal}}+\abs{T_{\p_y \Bcal}}\right)\notag\\
&\quad \les \norm{(\bar w,A)}_{Y_\tau}^2   \left( \delta^{-1}+    \norm{\bar w}_{Z_\tau} + (1+  \delta^{-1} ) \norm{(\bar w,A)}_{X_\tau}\right) +  \norm{\bar w}_{Z_\tau} \norm{\bar w}_{H_\tau} \norm{(\bar w,A)}_{Y_\tau}
\notag\\
&\qquad + \norm{(\bar w,A)}_{X_\tau}^2  \left( \eps^{-1} \delta^{-1} + \delta^{-1} +   \norm{\bar w}_{Z_\tau}+ (1+  \delta^{-1}  ) \norm{(\bar w,A)}_{X_\tau} \right)  
\,.
\label{eq:dw:error:total}
\end{align}
The implicit constants in \eqref{eq:A:error:total}, \eqref{eq:w:error:total}, and \eqref{eq:dw:error:total} only depend on $r$, since we have assumed $t , \eps \leq 1$.

\subsection{Proof of the Main Theorem}
In order to prove Theorem~\ref{thm:main} we couple together the energy estimates  \eqref{eq:dt:A:energy}, \eqref{eq:LeBron:3}, and \eqref{eq:Kobe:2} multiplied by the small factor $\delta^{-2}$, together with the error estimates \eqref{eq:A:error:total}, \eqref{eq:w:error:total}, and \eqref{eq:dw:error:total}, to obtain, for universal constants $C_0, \tilde{C}_0$, 
\begin{align}
&\frac{d}{2dt} \norm{(\bar w,A)}_{X_\tau}^2 + (-\dot \tau) \norm{(\bar w,A)}_{Y_\tau}^2 + \frac 18 \norm{\bar w}_{Z_\tau}^2 + \frac{1-\eps}{8 \eps (1+t/\eps)^2} \norm{\bar w}_{H_\tau}^2 
\notag\\
&\le C_0 \Big(  \delta^{-2}  \norm{\bar w}_{Z_\tau}^2   +   \norm{\bar w}_{Z_\tau} \norm{\bar w}_{H_\tau} \norm{(\bar w,A)}_{Y_\tau}
\notag\\
&\quad  + \norm{(\bar w,A)}_{Y_\tau}^2  \left( (\eps^{-1}  +  \delta^{-1}) +    \norm{\bar w}_{Z_\tau} +(1+ \delta^{-1}) \norm{(\bar w,A)}_{X_\tau}+ \delta   \norm{\bar w}_{H_\tau}  \right)
\notag\\ \n
&\quad  + \norm{(\bar w,A)}_{X_\tau}^2 \Big( ( \eps^{-1} +  \eps^{-1} \delta^{-1} +   \delta^{-1})+   \norm{\bar w}_{Z_\tau}+ (1+ \delta^{-1}) \norm{(\bar w,A)}_{X_\tau} + \delta    \norm{\bar w}_{H_\tau}    \Big) \\ \n
& \le \tilde{C}_0 \delta^{-2}  \norm{\bar w}_{Z_\tau}^2   + \frac{1}{100} \| \bar{w} \|_{Z_\tau}^2 + \tilde{C}_0 \| \bar{w} \|_{H_\tau}^2 \| (\bar{w}, A) \|_{Y_\tau}^2 
\notag\\
&\quad  +\tilde{C}_0 \norm{(\bar w,A)}_{Y_\tau}^2\left( ( \eps^{-1}   +  \delta^{-1}) +    \norm{\bar w}_{Z_\tau} + (1+ \delta^{-1}) \norm{(\bar w,A)}_{X_\tau}+ \delta   \norm{\bar w}_{H_\tau}  \right)
\notag\\ \label{diesel:1}
&\quad + \tilde{C}_0\norm{(\bar w,A)}_{X_\tau}^2 \left( (\eps^{-1}  + \eps^{-1}\delta^{-1} +   \delta^{-1})+   \norm{\bar w}_{Z_\tau}+ (1+ \delta^{-1}) \norm{(\bar w,A)}_{X_\tau} + \delta    \norm{\bar w}_{H_\tau}  \right).
\end{align}

\noindent To go from the second inequality to the final inequality, we have simply used Young's inequality for products to split the trilinear term, and denoted by $\tilde{C}_0$ the resulting (universal) constant. The constant $\tilde{C}_0$ is independent of the parameters $\delta, \eps$. We now take $\delta \gg 1$ so as to ensure that 
\begin{align} \label{choice:delta}
\frac{1}{100} + \frac{\tilde{C}_0}{\delta^{2}} \leq \frac{1}{16}, 
\end{align}

\noindent upon which the first two $\| \bar{w} \|_{Z_\tau}^2$ terms in \eqref{diesel:1} can be absorbed to the left-hand side. This yields the bound 
\begin{align} \n
&\frac{d}{2dt} \norm{(\bar w,A)}_{X_\tau}^2 + (-\dot \tau) \norm{(\bar w,A)}_{Y_\tau}^2 + \frac{1}{16} \norm{\bar w}_{Z_\tau}^2 + \frac{1-\eps}{8 \eps (1+t/\eps)^2} \norm{\bar w}_{H_\tau}^2 \\
&\quad \leq \tilde{C}_1  \| (\bar{w}, A) \|_{Y_\tau}^2 \Big( \| \bar{w} \|_{H_\tau}^2 +  \left( \eps^{-1} +     \norm{\bar w}_{Z_\tau} +  \norm{(\bar w,A)}_{X_\tau}+ \delta   \norm{\bar w}_{H_\tau}  \right) \Big)
\notag\\ \n
&\qquad \qquad + \tilde{C}_1 \norm{(\bar w,A)}_{X_\tau}^2  \left( \eps^{-1} +   \norm{\bar w}_{Z_\tau}+  \norm{(\bar w,A)}_{X_\tau} + \delta    \norm{\bar w}_{H_\tau}  \right),
\end{align}
for another universal constant $\tilde{C}_1$, which is again independent of $\eps$, and large values of $\delta$. 

By multiplying through by a sufficiently large universal constant, and taking $\eps \leq 1/64$, we obtain
\begin{align*}
&\frac{d}{dt} \norm{(\bar w,A)}_{X_\tau}^2 + (-\dot \tau) \norm{(\bar w,A)}_{Y_\tau}^2 + \frac{1}{16} \norm{\bar w}_{Z_\tau}^2 + \frac{1}{64 \eps } \norm{\bar w}_{H_\tau}^2 \\
&\quad \leq \Gamma_1(t) \| (\bar{w}, A) \|_{Y_\tau}^2 + \Gamma_2(t) \norm{(\bar w,A)}_{X_\tau}^2,
\end{align*}
where we have defined
\begin{subequations}
\begin{align} \label{def:gamma:1}
&\Gamma_1(t) := C_1 \Big( \| \bar{w} \|_{H_\tau}^2 + \left(  \eps^{-1} +   \frac 14  \norm{\bar w}_{Z_\tau} + \norm{(\bar w,A)}_{X_\tau}+ \delta   \norm{\bar w}_{H_\tau}  \right) \Big) \\ \label{def:gamma:2}
&\Gamma_2(t) := C_2 \Big(  \eps^{-1} +  \frac{1}{4} \norm{\bar w}_{Z_\tau}+  \norm{(\bar w,A)}_{X_\tau} + \delta    \norm{\bar w}_{H_\tau}  \Big).
\end{align}
\end{subequations}
for some universal constants $C_1, C_2 \geq 1$.
We now make the selection of 
\[
\dot{\tau} = - \Gamma_1 -1
\] 
from which the following integral identity and inequality follow: 
\begin{align} \label{Int:id}
&\tau(t) = \tau_0 - t - \int_0^t \Gamma_1(s) \ud s, \qquad E(t) \le E(0) +   \abs{ \int_0^t \Gamma_2(s) \| (\bar{w}, A) \|_{X_\tau(s)}^2 \ud s },
\end{align}
where the total analytic energy, $E(t)$, has been defined in \eqref{def:E}. The main result will now follow from:  
\begin{lemma}
\label{lem:eps:choice}
Fix the parameter $\delta$ according to \eqref{choice:delta}. There exist universal constants $C_1, C_2$ so that if the parameters $\eps$ and the time of existence $T_\ast$ satisfy simultaneously the inequalities 
\begin{subequations}
\label{eq:eps:choice}
\begin{align} \label{choice:eps:a}
&\frac{3}{2} C_1 \eps  E_0 \leq  \frac{\tau_0}{8}, \\ \label{choice:T:b}
&T_\ast^{\frac 1 4} C_2 (1 + \delta)  (T_\ast^{\frac 12} \eps^{-1} + 1) \le 1, \\ \label{choice:T:c}
&T_\ast^{\frac 1 2} C_1 (1 + \delta) (T_\ast^{\frac 1 2} + 1) \left(\frac{3}{2}E_0\right)^{\frac{1}{2}} \leq \frac{\tau_0}{8}, \\ \label{Tast.tau}
&T_\ast \leq \frac{\tau_0}{4} , \\
&T_\ast^{\frac 14} (1 + E_0^{\frac 14}) \leq \frac{1}{16} \, .
\end{align}
\end{subequations}
Then $|E(t)| \le \frac{3}{2}E(0)$ for all $t \in [0, T_\ast]$ and $\tau(t) \ge \frac{\tau_0}{2}$ for all $t \in [0, T_\ast]$. Moreover, it is possible to select the parameter $\eps$ and the time of existence $T_\ast$ so as to achieve the inequalities \eqref{eq:eps:choice}.
\end{lemma}
\begin{proof}[Proof of Lemma~\ref{lem:eps:choice}] We first establish, using the definition of $\Gamma_2$ in \eqref{def:gamma:2}, the following estimate  
\begin{align} \n
\abs{\int_0^{T_\ast} \Gamma_2(s) \ud s} \le & C_2 \int_0^{T_\ast} ( \eps^{-1}  +   \norm{\bar w}_{Z_\tau}+  \norm{(\bar w,A)}_{X_\tau} + \delta    \norm{\bar w}_{H_\tau} ) \ud s \\ \n
\le & C_2 \left(T_\ast  \eps^{- 1} + T_\ast^{\frac 1 2} \Big\| \| \bar{w} \|_{Z_\tau} \Big\|_{L^2(0, T_\ast)} + T_\ast \sup_{t \in [0, T_\ast]} \| (\bar{w}, A) \|_{X_\tau} + \delta T_\ast^{\frac 1 2} \Big\| \| \bar{w} \|_{H_\tau} \Big\|_{L^2(0, T_\ast)} \right) \\ \label{seq:1}
\le & C_2 (1 + \delta)  (T_\ast \eps^{-1} + T_\ast^{\frac 1 2})(1 + E(T_\ast)^{\frac 1 2} ).
\end{align}
Next, we establish using \eqref{def:gamma:1},
\begin{align} 
\abs{\int_0^{T_\ast} \Gamma_1(s) \ud s } \le & C_1 \abs{\int_0^{T_\ast} \Big( \| \bar{w} \|_{H_\tau}^2 + \left(  \eps^{-1} +   \frac 14  \norm{\bar w}_{Z_\tau} + \norm{(\bar w,A)}_{X_\tau}+ \delta   \norm{\bar w}_{H_\tau}  \right) \Big) \ud s }
\notag \\
\le & C_1 \eps  \Big\| \frac{1}{64 \eps} \| \bar{w} \|_{H_\tau} \Big\|_{L^2(0,T_\ast)}^2 + C_1 \Big(T_\ast  \eps^{-1} + T_\ast^{\frac 1 2} \Big\| \| \bar{w} \|_{Z_\tau} \Big\|_{L^2(0, T_\ast)} \notag  \\ 
&\qquad \qquad \qquad \qquad  \qquad \qquad + T_\ast \Big\| \| (\bar{w}, A) \|_{X_\tau} \Big\|_{L^\infty(0, T_\ast)} +  \delta T_\ast^{\frac 1 2} \Big\| \| \bar{w} \|_{H_\tau} \Big\|_{L^2(0, T_\ast)} \Big) \notag \\ 
\le & C_1 \eps  E(T_\ast) + C_1  (1 + \delta) (T_\ast \eps^{-1} + T_\ast^{\frac 1 2}) E(T_\ast)^{\frac{1}{2}}.
\label{seq:2}
\end{align}
We now select $\eps$ via 
\begin{align} \label{eps:choice}
\frac{3}{2} C_1 \eps  E_0 < \frac{\tau_0}{8}.
\end{align}
Once $\eps$ has been selected ({\em depending only on the initial datum}) through \eqref{eps:choice}, we pick $T_\ast$ depending on $\delta, \eps$ in order to satisfy simultaneously the two inequalities
\begin{align} \label{T:choice}
C_2 (1 + \delta)(T_\ast^{\frac 34} \eps^{-1}  + T_\ast^{\frac 1 4}) \le 1, \qquad C_1  (1 + \delta) (T_\ast \eps^{-1} + T_\ast^{\frac 1 2}) \left(\frac{3}{2}E_0\right)^{\frac{1}{2}} < \frac{\tau_0}{8}.
\end{align}
Such a choice is possible because every parameter other than $T_\ast$ in \eqref{T:choice} has been fixed already, so we can take $T_\ast$ small enough so as to achieve \eqref{T:choice}.

Inserting the first inequality in \eqref{T:choice} into the second integral inequality in \eqref{Int:id}, we obtain the nonlinear inequality 
\begin{align*}
E(T_\ast) &\le  E_0 +  \sup_{t \in [0, T_\ast]} \| (\bar{w}, A) \|_{X_{\tau(s)}} \|^2 \int_0^{T_\ast} \Gamma_2(s) \ud s| \\
&\le E_0 + C_2 (1 + \delta) (T_\ast  \eps^{-1} + T_\ast^{\frac 1 2})(1 + E(T_\ast)^{\frac  1 2} ) E(T_\ast) \\
&\le   E_0 +  T_\ast^{\frac 1 4} E(T_\ast) +  T_\ast^{\frac 1 4} E(T_\ast)^{\frac 3 2},
\end{align*} 
which implies the desired bound, $E(T_\ast) \le \frac 3 2 E_0$ by selecting $T_\ast$ small enough to obey
\[
T_\ast^{\frac 14} (1 + E_0^{\frac 14}) \leq \frac{1}{16} \, .
\]
We subsequently insert the inequality $E(T_\ast) \le \frac 3 2 E_0$ together with the two inequalities \eqref{eps:choice} and the second inequality of \eqref{T:choice} into \eqref{seq:2} so as to ensure for all $t \in [0, T_\ast]$
\begin{align} \label{ineq:tau}
|\tau(t)| \ge \tau_0 - |t| - \abs{\int_0^{t} \Gamma_1(s) \ud s} \ge \tau_0 -T_\ast - \frac{\tau_0}{8} - \frac{\tau_0}{8} \ge \frac{\tau_0}{2},
\end{align}
where we have appealed to the last inequality on $T_\ast$, \eqref{Tast.tau}, to establish the final inequality in \eqref{ineq:tau}. In summary, one first chooses $\delta$, then $\eps$, and $T_\ast$ is picked last.
\end{proof}

\section{Proof of the Main Nonlinear Lemma}
\label{sec:main:lemma}

Before turning to the proof, we recall a few technical results which are used in the proof of Lemma~\ref{lem:main}.

\subsection{Properties the weight and the lift function}
Let us record the following straightforward inequality
\begin{align} \label{Iy:theta:est}
\abs{I_y[\theta_\xi]} =  \int_0^y \theta_\xi(y', t) \ud y'  \le \int_0^y 1 \ud y' \le  y, 
\end{align}
and emphasize that \eqref{Iy:theta:est} is \textit{independent} of the parameter $\eps$. Recalling \eqref{eq:rho:def}, we note that $\theta_\xi$ obeys
\begin{align}
(1-\theta_\xi) \rho = e^{- \frac{y^2 \langle \xi \rangle^2}{2(1 + \frac{t}{\eps} )}} e^{\frac{y^2}{8(1 + \frac t \eps)}} = e^{-\frac{y^2(4\langle \xi \rangle^2 - 1)}{8(1 + \frac{t}{\eps})}} \leq e^{-\frac{3 y^2 \langle \xi \rangle^2}{8(1 + \frac{t}{\eps})}}   
\label{eq:weight:choice}
\end{align}
from which we may deduce the the inequality
\begin{align}
 \label{theta:rho:bound}
&\| (1-\theta_\xi) \rho \|_{L^2_y} \lesssim \frac{(1 +  t/\eps  )^{1/4}}{\langle \xi \rangle^{1/2}}  \lesssim \frac{1}{\langle \xi \rangle^{1/2}} ,
\end{align}
which will be used frequently below.

It will be convenient to appeal to the bound
\begin{align}
\norm{I_y[f]}_{L^\infty_y} \les (1+t/\eps)^{1/4} \norm{\rho f}_{L^2([0,\infty))} \les  \norm{\rho f}_{L^2([0,\infty))} 
\label{eq:v:L:infty}
\end{align}
which is a consequence of the estimate $\norm{\rho^{-1}}_{L^2(0,\infty)} \les (1+t/\eps)^{1/4}$ and 
\begin{align*}
\norm{I_y[f]}_{L^\infty_y}  \leq \int_0^\infty \rho(y) f(y) \rho(y)^{-1} dy \leq \norm{\rho f}_{L^2([0,\infty))}  \norm{\rho^{-1}}_{L^2[0,\infty)} .
\end{align*}
As a consequence of the proof of \eqref{eq:v:L:infty}, the fundamental theorem of calculus, we have that whenever $f|_{y=0}=0$, the estimate
\begin{align}
\norm{f}_{L^\infty_y} \leq \norm{I_y[\p_y f]}_{L^\infty_y} \les (1+t/\eps)^{1/4} \norm{\rho \p_y f}_{L^2([0,\infty))} \les  \norm{\rho \p_y f}_{L^2([0,\infty))}
\label{eq:w:L:infty}
\end{align}
also holds. The following weighted Poincar\'e/Hardy inequality will be useful for our proof. 
\begin{lemma}
\label{lem:Hardy}
Let the weight function $\rho$ be as defined in \eqref{eq:rho:def}. Assume that $f$ is such that $\rho \p_y f \in L^2$. Then we have
\begin{align*}
\norm{\rho f}_{L^2}^2 + \frac{1}{4(1+t/\eps)} \norm{y \rho f}_{L^2}^2 \leq 4(1+t/\eps) \norm{\rho \partial_y f}_{L^2}^2. 
\end{align*}
\end{lemma}
\begin{proof}[Proof of Lemma~\ref{lem:Hardy}]
\begin{align*}  
\int_0^\infty \rho^2 f^2 =  \int_0^\infty e^{\frac{y^2}{4(1+t/\eps)}} f^2 &=  \int_0^\infty \p_y \{ y \} e^{\frac{y^2}{4(1+t/\eps)}} f^2 \notag \\
&= \frac{-1}{2(1+t/\eps)}  \int_0^\infty y^2  e^{\frac{y^2}{4(1+t/\eps)}} f^2 -2  \int_0^\infty y e^{\frac{y^2}{4(1+t/\eps)}} f \p_y f. 
\end{align*}
Due to the monotone increasing and super-exponential nature of our weight, the negative term on the right side of the above is the key contribution. Rearranging yields 
\begin{align*}
\norm{\rho f}_{L^2}^2 + \frac{1}{2(1+t/\eps)} \norm{y \rho f}_{L^2}^2 
&\leq 2 \left| \int_0^\infty y \rho^2 f \p_y f\right|  \leq 2 \norm{y \rho f}_{L^2} \norm{\rho \p_y f}_{L^2} \notag\\
& \leq \frac{1}{4(1+t/\eps)} \norm{y \rho f}_{L^2}^2 + 4 (1+t/\eps) \norm{\rho \p_y f}_{L^2}^2 \, ,
\end{align*}
which concludes the proof of the lemma.
\end{proof}

\subsection{Error terms in the $A$ energy}
First, using \eqref{eq:v:L:infty} and the Cauchy-Schwartz inequality it follows that the term $T_{\mathcal A,1}$ defined in \eqref{eq:T:A12:a} may be bounded as
\begin{align*}
\abs{T_{\mathcal A,1}}
&\leq \norm{|\xi|^{1/2} \brak{\xi}^r e^{\tau |\xi|} I_\infty[\bar w_\xi]}_{L^2_\xi L^\infty_y} \norm{|\xi|^{1/2} \brak{\xi}^r e^{\tau |\xi|} A_\xi}_{L^2_\xi} 
 \les \norm{|\p_x|^{1/2} \bar w}_{\tau,r}  \norm{|\p_x|^{1/2} A}_{\tilde{\tau,r}}
\,. 
\end{align*}
In order to estimate the term $T_{\mathcal{A},2}$ defined in \eqref{eq:T:A12:b}
we use that for any $r\geq 0$ we have
\begin{align*}
\brak{\xi}^r \les \brak{\xi-\eta}^r + \brak{\eta}^r 
\end{align*}
where the implicit constant depends solely on $r$, and that $\brak{\xi}^{1/2} \leq \brak{\eta}^{1/2} \brak{\xi-\eta}^{1/2}$, to conclude that
\begin{align*}
\abs{T_{\mathcal A,2}} 
&\les \int_{\RR} \int_{\RR} \abs{\xi-\eta} \left(\brak{\xi-\eta}^{r-1/2} + \brak{\eta}^{r-1/2}\right) \abs{A_\eta} e^{\tau |\eta|} \abs{A_{\xi-\eta}} e^{\tau |\xi-\eta|} \brak{\xi}^{r+1/2} \abs{A_\xi} e^{\tau|\xi|} \ud \eta \ud \xi
\notag\\
&\les \norm{|\p_x|^{1/2} A}_{\tilde{\tau,r}} \norm{A_\eta e^{\tau |\eta|}}_{L^1_\eta} \norm{\brak{\p_x}^{1/2} A}_{\tilde{\tau,r}} 
 + \norm{|\xi-\eta| \brak{\xi-\eta}^{1/2} A_{\xi-\eta} e^{\tau |\xi-\eta|}}_{L^1_{\xi-\eta}} \norm{A}_{\tilde{\tau,r}}^2
\notag\\
&\les \norm{|\p_x|^{1/2} A}_{\tilde{\tau,r}}^2 \norm{A}_{\tilde{\tau,r}}  + \norm{|\p_x|^{1/2} A}_{\tilde{\tau,r}} \norm{A}_{\tilde{\tau,r}}^2.
\end{align*}
In the last inequality above we have used that for $r>3/2$ we have $\brak{\xi}^{-r+1} \in L^2_\xi$. This proves~\eqref{eq:Shaq:A}.

\subsection{Error terms in the $\bar w$ energy}
\label{sec:bar:w:energy}

\subsubsection{The $T_{\Ncal}$ term}
According to \eqref{eq:N:def} and \eqref{eq:T:cal:1}, we decompose $T_{\Ncal}$ as
\begin{align}
T_{\Ncal} 
&= i \int_{\RR} \int_{\RR} \int_0^\infty \bar w_\eta (\xi-\eta) \bar w_{\xi-\eta}  \overline{\bar w_\xi} \rho^2 e^{2\tau|\xi|} \brak{\xi}^{2r} \ud y \ud\eta\ud \xi
\notag\\
&\quad- i \int_{\RR} \int_{\RR}  \int_0^\infty  \eta I_y[\bar w_\eta] \partial_y \bar w_{\xi-\eta} \overline{\bar w_\xi} \rho^2 e^{2\tau|\xi|} \brak{\xi}^{2r}  \ud y \ud\eta\ud \xi
\notag\\
&=: T_{\Ncal}^{(1)} + T_{\Ncal}^{(2)}
\,.
\label{eq:Barkley:0}
\end{align}
For the first term above we appeal to the inequality 
\[
|\xi-\eta|^{1/2} \brak{\xi}^r \les (|\eta|^{1/2} + |\xi|^{1/2}) \left(\brak{\eta}^r + \brak{\xi-\eta}^r\right),
\] 
to the triangle inequality of the exponential term, and to the bound \eqref{eq:w:L:infty} to conclude
\begin{align}
 \abs{T_{\Ncal}^{(1)}}
 &\les \int_{\RR} \int_{\RR}    |\eta|^{1/2} \norm{\rho \bar w_\eta}_{L^2_y}  |\xi-\eta|^{1/2}\norm{\rho\bar w_{\xi-\eta}}_{L^2_y}  \left(\brak{\eta}^r + \brak{\xi-\eta}^r\right) \brak{\xi}^{r} \norm{\bar w_\xi}_{L^\infty_y}  e^{2\tau|\xi|}    \ud\eta\ud \xi
 \notag\\
 &\quad +  \int_{\RR} \int_{\RR}   \norm{\bar w_\eta}_{L^\infty_y}  |\xi-\eta|^{1/2}\norm{\rho \bar w_{\xi-\eta}}_{L^2_y}  \left(\brak{\eta}^r + \brak{\xi-\eta}^r\right)  |\xi|^{1/2} \brak{\xi}^{r} \norm{\rho \bar w_\xi}_{L^2_y}  e^{2\tau|\xi|}   \ud\eta\ud \xi
 \notag\\
&\les  \norm{|\p_x|^{1/2} \bar w}_{\tau,r} \norm{\p_y \bar w}_{\tau,r} \norm{|\zeta|^{1/2} \rho \bar w_{\zeta} e^{\tau|\zeta|}}_{L^1_{\zeta} L^2_y} 
  +  \norm{|\p_x|^{1/2} \bar w}_{\tau,r}^2  \norm{ \rho \p_y \bar w_{\zeta} e^{\tau|\zeta|}}_{L^1_{\zeta} L^2_y} 
\notag\\
&\les  \norm{|\p_x|^{1/2} \bar w}_{\tau,r}^2  \norm{\p_y \bar w}_{\tau,r}
\label{eq:Barkley:1}
\end{align}
where in the last inequality we have used that $r>1/2$.

For the second term in \eqref{eq:Barkley:0} we proceed similarly, but appeal to the bound \eqref{eq:v:L:infty}, which yields 
\begin{align}
 \abs{T_{\Ncal}^{(2)}}
 &\les \int_{\RR} \int_{\RR}    |\eta|^{1/2} \brak{\eta}^r \norm{I_y[\bar w_\eta]}_{L^\infty_y}  |\xi-\eta|^{1/2}\norm{\rho \p_y \bar w_{\xi-\eta}}_{L^2_y}     \brak{\xi}^{r} \norm{\rho \bar w_\xi}_{L^2_y}  e^{2\tau|\xi|}    \ud\eta\ud \xi
 \notag\\
  &\quad +  \int_{\RR} \int_{\RR}   |\eta|^{1/2} \brak{\eta}^r \norm{I_y[\bar w_\eta]}_{L^\infty_y} \norm{\rho \p_y \bar w_{\xi-\eta}}_{L^2_y}     |\xi|^{1/2} \brak{\xi}^{r} \norm{\rho \bar w_\xi}_{L^2_y}  e^{2\tau|\xi|}   \ud\eta\ud \xi
 \notag\\
 &\quad + \int_{\RR} \int_{\RR}    |\eta|  \norm{I_y[\bar w_\eta]}_{L^\infty_y}  \brak{\xi-\eta}^r \norm{\rho \p_y \bar w_{\xi-\eta}}_{L^2_y}   \brak{\xi}^{r} \norm{\rho \bar w_\xi}_{L^2_y}  e^{2\tau|\xi|}    \ud\eta\ud \xi
 \notag\\
&\les  \norm{|\p_x|^{1/2} \bar w}_{\tau,r} \norm{\bar w}_{\tau,r}\norm{|\xi-\eta|^{1/2} \rho \p_y \bar w_{\xi-\eta} e^{\tau|\xi-\eta|}}_{L^1_{\xi-\eta} L^2_y} 
\notag\\
&\quad +  \norm{|\p_x|^{1/2} \bar w}_{\tau,r}^2 \norm{\rho \p_y \bar w_{\xi-\eta} e^{\tau|\xi-\eta|}}_{L^1_{\xi-\eta} L^2_y} +  \norm{\bar w}_{\tau,r} \norm{\p_y \bar w}_{\tau,r} \norm{|\eta| \rho \bar w_{\eta} e^{\tau|\eta|}}_{L^1_{\eta} L^2_y} 
\notag\\
&\les  \norm{|\p_x|^{1/2} \bar w}_{\tau,r} \norm{\brak{\p_x}^{1/2} \bar w}_{\tau,r}  \norm{\p_y \bar w}_{\tau,r}
 \label{eq:Barkley:2}
\end{align}
since $r>1$. Combining the estimates \eqref{eq:Barkley:1}--\eqref{eq:Barkley:2} yields the desired bound \eqref{eq:Barkley:*}.

\subsubsection{The $T_{\Lcal}$ term} Recall that the term $T_{\Lcal}$ is defined in \eqref{eq:T:cal:1}, via~\eqref{eq:NLMB:def}, as
\begin{align*}
T_{\Lcal}
&= i \int_{\RR} \int_{\RR}  \int_0^\infty   \rho^2 e^{2 \tau |\xi|} \langle \xi \rangle^{2r} \overline{\bar{w}_\xi} \eta \biggl( \bar w_{\xi-\eta}  A_{\eta} \theta_{\eta} + A_{\xi-\eta} \theta_{\xi-\eta}   \bar w_{\eta}     -   I_y[\bar w_\eta] A_{\xi-\eta} \partial_y \theta_{\xi-\eta} \biggr) \ud y \ud \eta \ud \xi  
\notag \\
&=  T_{\Lcal}^{(1)} + T_{\Lcal}^{(2)} + T_{\Lcal}^{(3)} 
\,. 
\end{align*}
We estimate each of the above terms individually. Using that $0\leq \theta_\xi \leq 1$ pointwise in $t,y,\xi$, that $|\eta|^{1/2} \leq |\xi|^{1/2} + |\xi-\eta|^{1/2}$, and $\brak{\xi}^r \les \brak{\eta}^r + \brak{\xi-\eta}^r$, for the $T_{\Lcal}^{(1)}$ term we have
\begin{align}
\abs{T_{\Lcal}^{(1)}} 
&\le \int_{\RR} \int_{\RR}  \int_0^\infty   \rho^2 e^{\tau |\xi|} e^{\tau |\xi-\eta|} e^{\tau |\eta|}  \brak{\xi}^{2r} \abs{\bar{w}_\xi} |\eta|^{1/2} (|\xi|^{1/2} + |\xi-\eta|^{1/2})  \abs{\bar{w}_{\xi-\eta}} \abs{A_{\eta}}    \ud y  \ud \eta \ud \xi  
\notag \\
&\les \norm{|\p_x|^{1/2} \bar w}_{\tau,r} \norm{|\p_x|^{1/2} A}_{\tilde{\tau,r}} \norm{\rho e^{\tau |\xi-\eta|}  \abs{\bar w_{\xi-\eta}}}_{L^1_{\xi-\eta} L^2_y} 
\notag \\
&\qquad \qquad + \norm{\bar w}_{\tau,r} \norm{|\p_x|^{1/2} A}_{\tilde{\tau,r}} \norm{\rho |\xi-\eta|^{1/2} e^{\tau |\xi-\eta|}  \abs{\bar w_{\xi-\eta}}}_{L^1_{\xi-\eta} L^2_y} 
\notag \\
&\qquad \qquad + \norm{|\p_x|^{1/2}  \bar w}_{\tau,r}  \norm{ \bar w}_{\tau,r}   \norm{|\eta|^{1/2} e^{\tau |\eta|}  \abs{A_\eta}}_{L^1_{\eta}}
\notag \\
&\les  \norm{|\p_x|^{1/2} \bar w}_{\tau,r} \norm{\bar w}_{\tau,r} \norm{|\p_x|^{1/2} A}_{\tilde{\tau,r}}
\label{eq:Jordan:1}
\end{align}
since $r>1$. Similarly, for the $T_{\Lcal}^{(2)}$ term we have
\begin{align}
\abs{T_{\Lcal}^{(2)}} 
&\le \int_{\RR} \int_{\RR}  \int_0^\infty   \rho^2 e^{\tau |\xi|} e^{\tau |\xi-\eta|} e^{\tau |\eta|}  \brak{\xi}^{2r} \abs{\bar{w}_\xi} |\eta|^{1/2} (|\xi|^{1/2} + |\xi-\eta|^{1/2})  \abs{\bar{w}_{\eta}} \abs{A_{\xi-\eta}}    \ud y  \ud \eta \ud \xi 
\notag \\
&\les \norm{|\p_x|^{1/2} \bar w}_{\tau,r}^2  \norm{e^{\tau |\xi-\eta|}  \abs{A_{\xi-\eta}}}_{L^1_{\xi-\eta}} 
 + \norm{|\p_x|^{1/2} \bar w}_{\tau,r} \norm{\bar w}_{\tau,r}  \norm{|\xi-\eta|^{1/2} e^{\tau |\xi-\eta|}  \abs{A_{\xi-\eta}}}_{L^1_{\xi-\eta}} 
\notag \\
&\qquad \qquad + \left( \norm{|\p_x|^{1/2}  \bar w}_{\tau,r}  \norm{A}_{\tilde{\tau,r}} +\norm{\bar w}_{\tau,r}  \norm{|\p_x|^{1/2} A}_{\tilde{\tau,r}}  \right)  \norm{\rho |\eta|^{1/2} e^{\tau |\eta|}  \abs{\bar w_\eta}}_{L^1_{\eta}L^2_y}
\notag \\
&\les \norm{|\p_x|^{1/2} \bar w}_{\tau,r}^2 \norm{A}_{\tilde{\tau,r}} +  \norm{|\p_x|^{1/2} \bar w}_{\tau,r} \norm{\bar w}_{\tau,r} \norm{|\p_x|^{1/2} A}_{\tilde{\tau,r}}
\label{eq:Jordan:2}
\, .
\end{align}
The $T_{\Lcal}^{(3)}$ term is treated slightly differently due to the presence of $\partial_y \theta_{\xi-\eta}$. Here we use that 
\begin{align*}
\rho \brak{\xi}^{-1/2} \p_y \theta_{\xi}
=\frac{\rho(y,t) \brak{\xi}^{3/2} y}{ (1+t/\eps)} e^{-\frac{y^2 \brak{\xi}^2}{2(1+t/\eps)}}
&\les \frac{\brak{\xi}^{1/2}}{(1+t/\eps)^{1/2}} e^{\frac{y^2}{8(1+t/\eps)}}  e^{-\frac{y^2 \brak{\xi}^2}{4(1+t/\eps)}}
\notag\\
&\les \frac{\brak{\xi}^{1/2}}{(1+t/\eps)^{1/2}}  e^{-\frac{y^2 \brak{\xi}^2}{8(1+t/\eps)}} \les \brak{\xi}^{1/2}  e^{-\frac{y^2 \brak{\xi}^2}{8(1+t/\eps)}}
\end{align*}
from which it follows upon taking an $L^2$ norm in $y$ and a supremum over $\xi\in \RR$ that
\begin{align}
 \norm{\rho \brak{\xi}^{-1/2} \p_y \theta_{\xi}}_{L^\infty_{\xi} L^2_y} 
  \les 1 \,.
 \label{eq:Jordan:0}
\end{align}
Appealing also to \eqref{eq:v:L:infty} with $f=\bar w_\eta$, 
and to the inequality
\begin{align*}
|\eta| \brak{\xi-\eta}^{1/2} \brak{\xi}^r \les  |\eta| \brak{\xi-\eta}^{r+1/2} + |\eta|^{1/2} \brak{\eta}^r |\xi|^{1/2} \brak{\xi-\eta}^{1/2} + |\eta|^{1/2} \brak{\eta}^r |\xi-\eta|^{1/2} \brak{\xi-\eta}^{1/2}
\end{align*}
we obtain
\begin{align}
\abs{T_{\Lcal}^{(3)}} 
&\leq  \int_{\RR} \int_{\RR}  \int_0^\infty \rho  e^{\tau |\xi|} e^{\tau |\xi-\eta|} e^{\tau |\eta|}  \brak{\xi}^{2r} \abs{\bar{w}_\xi} |\eta| \brak{\xi-\eta}^{1/2} \abs{A_{\xi-\eta}}  \abs{I_y[\bar{w}_{\eta}]}  \frac{\rho \abs{\p_y \theta_{\xi-\eta}}}{\brak{\xi-\eta}^{1/2}}    \ud y  \ud \eta \ud \xi 
\notag\\
&\les \int_{\RR} \int_{\RR} \norm{\rho \bar w_\xi}_{L^2_y} \brak{\xi}^r e^{\tau |\xi|} \abs{A_{\xi-\eta}} e^{\tau|\xi-\eta|} \norm{\rho \bar w_\eta}_{L^2_y} e^{\tau |\eta|} |\eta| \brak{\xi-\eta}^{1/2} \brak{\xi}^r \ud \eta \ud \xi
\notag\\
&\les \norm{\bar w}_{\tau,r} \norm{\brak{\p_x}^{1/2} A}_{\tilde{\tau,r}} \norm{|\eta| \norm{\rho \bar w_\eta}_{L^2_y} e^{\tau |\eta|}}_{L^1_\eta}
+ \norm{|\p_x|^{1/2} \bar w}_{\tau,r}^2 \norm{\brak{\xi-\eta}^{1/2} A_{\xi-\eta} e^{\tau |\xi-\eta|}}_{L^1_{\xi-\eta}}
\notag\\
&\qquad \qquad + \norm{\bar w}_{\tau,r} \norm{|\p_x|^{1/2} \bar w}_{\tau,r} \norm{\brak{\xi-\eta} A_{\xi-\eta} e^{\tau |\xi-\eta|}}_{L^1_{\xi-\eta}}
\notag\\
&\les \norm{\bar w}_{\tau,r}^2 \norm{\brak{\p_x}^{1/2} A}_{\tilde{\tau,r}}  
+ \norm{|\p_x|^{1/2} \bar w}_{\tau,r}^2 \norm{A}_{\tilde{\tau,r}} + \norm{\bar w}_{\tau,r} \norm{|\p_x|^{1/2} \bar w}_{\tau,r} \norm{A}_{\tilde{\tau,r}}
\label{eq:Jordan:3}
\end{align}
since $r> 3/2$. Summing \eqref{eq:Jordan:1}, \eqref{eq:Jordan:2}, and \eqref{eq:Jordan:3}, and massaging the resulting terms we arrive at \eqref{eq:Jordan:*}.

\subsubsection{The $T_{\Mcal}$ term}
This term has the distinguished feature of losing a $y$-weight, which is why we have introduced the vorticity $\p_y \bar w$ in the first place. Recall from \eqref{eq:M:def} and \eqref{eq:T:cal:1} that
\begin{align*}
 T_{\Mcal} = -i \int_{\RR} \int_{\RR}  \int_0^\infty   \rho^2 e^{2 \tau |\xi|} \langle \xi \rangle^{2r} \overline{\bar{w}_\xi} \eta A_\eta I_y[\theta_\eta] \p_y \bar w_{\xi-\eta}  \ud y \ud \eta \ud \xi  
 \, .
\end{align*}
Recall from \eqref{Iy:theta:est} that $\abs{I_y [\theta_\eta]} \leq y$, and thus, with the cutoff $\chi$ defined in \eqref{eq:chi:def} we have that 
\begin{align*}
 \abs{(1-\chi(y)) I_y [\theta_\eta]} \les 1
\end{align*}
pointwise in $\eta$. Therefore, the contribution to $T_{\Mcal}$ coming from the support of $1-\chi(y)$ may be bounded as
\begin{align}
&\abs{\int_{\RR} \int_{\RR}  \int_0^\infty   \rho^2 e^{2 \tau |\xi|} \langle \xi \rangle^{2r} \overline{\bar{w}_\xi} \eta A_\eta I_y[\theta_\eta] \p_y \bar w_{\xi-\eta} (1-\chi(y)) \ud y \ud \eta \ud \xi  } 
\notag\\
& \les \int_{\RR} \int_{\RR}  \int_0^\infty     e^{\tau |\xi|}e^{\tau |\eta|}e^{\tau |\xi-\eta|} \langle \xi \rangle^{2r} \rho \abs{\bar{w}_\xi} |\eta| \abs{A_\eta} \rho \abs{\p_y \bar w_{\xi-\eta}} \ud y \ud \eta \ud \xi  
\notag\\
&\les \norm{|\p_x|^{1/2} \bar w}_{\tau,r} \norm{|\p_x|^{1/2} A}_{\tilde{\tau,r}} \norm{\rho \abs{\p_y \bar w_{\xi-\eta}}e^{\tau |\xi-\eta|} }_{L^1_{\xi-\eta} L^2_y} 
+ \norm{\brak{\p_x}^{1/2} \bar w}_{\tau,r} \norm{\p_y \bar w}_{\tau,r} \norm{|\eta|^{1/2} A_\eta e^{\tau|\eta|}}_{L^1_\eta} 
\notag\\
&\qquad +  \norm{\brak{\p_x}^{1/2} \bar w}_{\tau,r} \norm{A}_{\tilde{\tau,r}} \norm{\rho |\xi-\eta|^{1/2} \abs{\p_y \bar w_{\xi-\eta}}e^{\tau |\xi-\eta|} }_{L^1_{\xi-\eta} L^2_y} 
\notag\\
&\les \norm{\brak{\p_x}^{1/2} \bar w}_{\tau,r} \norm{\p_y \bar w}_{\tau,r} \norm{\brak{\p_x}^{1/2} A}_{\tilde{\tau,r}} 
\,.
\label{eq:Jordan:4}
\end{align}
On the other hand, the contribution to $T_{\Mcal}$ from the support of $\chi(y)$ requires us to use the third term on the left side of \eqref{eq:Kobe:2}, but with $r$ replaced by $r-1/2$. More precisely, we use that $\chi(y) \abs{I_y[\theta_\eta]} \leq y \chi(y)$ pointwise in $\eta$, and the inequality 
\begin{align*}
 |\eta| \brak{\xi}^{2r} \les |\eta| \brak{\eta}^{r-1/2} \brak{\xi}^{r+1/2} + |\eta| \brak{\xi-\eta}^{r-1/2} \brak{\xi}^{r+1/2}
\end{align*}
to deduce that
\begin{align}
 &\abs{\int_{\RR} \int_{\RR}  \int_0^\infty   \rho^2 e^{2 \tau |\xi|} \langle \xi \rangle^{2r} \overline{\bar{w}_\xi} \eta A_\eta I_y[\theta_\eta] \p_y \bar w_{\xi-\eta} \chi(y) \ud y \ud \eta \ud \xi  } 
\notag\\
& \les \int_{\RR} \int_{\RR}    e^{\tau |\xi|}e^{\tau |\eta|}e^{\tau |\xi-\eta|} \langle \xi \rangle^{2r} \norm{\rho \bar{w}_\xi}_{L^2_y}  |\eta| \abs{A_\eta} \norm{\rho y \chi(y) \p_y \bar w_{\xi-\eta}}_{L^2_y}  \ud \eta \ud \xi  
\notag\\
&\les \norm{\brak{\p_x}^{1/2} \bar w}_{\tau,r} \left( \norm{\abs{\p_x}^{1/2} A}_{\tilde{\tau,r}} \norm{ \rho y \chi(y) \p_y \bar w_{\xi-\eta} e^{\tau |\xi-\eta|}}_{L^1_{\xi-\eta} L^2_y}  + \norm{y \chi \p_y \bar w}_{\tau,r-1/2} \norm{|\eta| A_\eta e^{\tau|\eta|}}_{L^1_\eta} \right) \notag\\
&\les \norm{\brak{\p_x}^{1/2} \bar w}_{\tau,r} \norm{y \chi \p_y \bar w}_{\tau,r-1/2}   
 \norm{\abs{\p_x}^{1/2} A}_{\tilde{\tau,r}} 
 \label{eq:Jordan:5}
\end{align}
since $r$ was taken to be sufficiently large. Combining \eqref{eq:Jordan:4} and \eqref{eq:Jordan:5} we obtain the estimate \eqref{eq:Jordan:**}.

\subsubsection{The $T_{\Bcal}$ term}
According to \eqref{eq:B:def:*} and \eqref{eq:T:cal:1}, we decompose the $T_{\Bcal}$ term as
\begin{align}
T_{\Bcal}
&= \int_{\RR}  \int_0^\infty A_\xi (\partial_t - \partial_{yy})\theta_\xi  \overline{\bar w_\xi} \rho^2 e^{2\tau|\xi|} \brak{\xi}^{2r} \ud y   \ud \xi 
\notag\\
&\qquad + \int_{\RR}   \int_0^\infty \left(\theta_\xi -1\right) \partial_t A_\xi  \overline{\bar w_\xi} \rho^2 e^{2\tau|\xi|} \brak{\xi}^{2r} \ud y   \ud \xi 
\notag\\
&\qquad + i \int_{\RR}  \int_0^\infty \xi \left( I_\infty[\bar w_\xi] - I_y[\bar w_\xi] \right) \overline{\bar w_\xi} \rho^2 e^{2\tau|\xi|} \brak{\xi}^{2r} \ud y   \ud \xi 
\notag\\
&\qquad +i \int_{\RR}  \int_0^\infty \xi   A_\xi \bigl( y (\theta_\xi-1)  - ( I_\infty[1-\theta_\xi] - I_y[1-\theta_\xi] )\bigr)  \overline{\bar w_\xi} \rho^2 e^{2\tau|\xi|} \brak{\xi}^{2r} \ud y   \ud \xi 
\notag\\
&\qquad +i \int_{\RR} \int_{\RR} \int_0^\infty \xi A_\eta  A_{\xi-\eta} \left( \theta_\eta \theta_{\xi-\eta} -1\right)  \overline{\bar w_\xi} \rho^2 e^{2\tau|\xi|} \brak{\xi}^{2r} \ud y \ud \eta \ud \xi 
\notag\\
&\qquad -i \int_{\RR} \int_{\RR} \int_0^\infty \eta A_\eta A_{\xi-\eta} I_y[\theta_\eta] \partial_y \theta_{\xi-\eta}  \overline{\bar w_\xi} \rho^2 e^{2\tau|\xi|} \brak{\xi}^{2r} \ud y \ud \eta \ud \xi 
\notag\\
&= T_{\Bcal}^{(1)} + T_{\Bcal}^{(2)} + T_{\Bcal}^{(3)} + T_{\Bcal}^{(4)} + T_{\Bcal}^{(5)} + T_{\Bcal}^{(6)}
\label{eq:Magic:0}
\end{align}
We bound the six terms above individually, and note that the second term, $T_{\Bcal}^{(2)}$, is the most involved one, as it involves analyzing the four terms arising from the $A_\xi$ evolution in \eqref{eq:A:evo}.

We bound the most difficult term first. Combining \eqref{eq:Magic:0} and \eqref{eq:A:evo} we rewrite
\begin{align}
T_{\Bcal}^{(2)} 
&=  -i \int_{\RR}   \int_0^\infty \left(\theta_\xi -1\right) \xi c_{\theta,\xi}  A_\xi \overline{\bar w_\xi} \rho^2 e^{2\tau|\xi|} \brak{\xi}^{2r} \ud y   \ud \xi \notag\\
&\qquad -i \int_{\RR}   \int_0^\infty \left(\theta_\xi -1\right)  \xi |\xi|  A_\xi \overline{\bar w_\xi} \rho^2 e^{2\tau|\xi|} \brak{\xi}^{2r} \ud y   \ud \xi \notag\\
&\qquad + i \int_{\RR}   \int_0^\infty \left(\theta_\xi -1\right) \xi I_\infty[\bar w_\xi] \overline{\bar w_\xi} \rho^2 e^{2\tau|\xi|} \brak{\xi}^{2r} \ud y   \ud \xi 
\notag\\
&\qquad -i  \int_{\RR} \int_{\RR}  \int_0^\infty \left(\theta_\xi -1\right) A_\eta A_{\xi-\eta} (\xi-\eta)  \overline{\bar w_\xi} \rho^2 e^{2\tau|\xi|} \brak{\xi}^{2r} \ud y \ud\eta  \ud \xi 
\notag\\
&= T_{\Bcal}^{(2,1)} + T_{\Bcal}^{(2,2)} + T_{\Bcal}^{(2,3)} + T_{\Bcal}^{(2,4)} \,.
\label{eq:Magic:2:0}
\end{align}
Using the definition of $c_{\theta,\xi}$ in \eqref{c:theta:xi} and the bound \eqref{theta:rho:bound} for the $L^2_y$ norm of $\rho \left(\theta_\xi -1\right)$, we   obtain
\begin{align}
\abs{T_{\Bcal}^{(2,1)}} \les  \norm{\bar w}_{\tau,r} \norm{A}_{\tilde{\tau,r}}\,.
\label{eq:Magic:2:1}
\end{align}
For the second term, we use the one dimensional Hardy inequality $\norm{f/y}_{L^2_y} \les \norm{\p_y f}_{L^2_y} \les \norm{\rho \p_y f}_{L^2_y}$, valid for $f$ such that $f|_{y=0}$ since $\rho \geq 1$, and the bound
\begin{align}
\brak{\xi}^{3/2} \norm{y (\theta_\xi-1) \rho^2}_{L^2_y} 
&\les \brak{\xi}^{3/2} \norm{ y e^{-\frac{y^2 \brak{\xi}^2}{4(1+t/\eps)}}}_{L^2_y} \les   \brak{\xi}^{1/2}  \norm{e^{-\frac{y^2 \brak{\xi}^2}{8(1+t/\eps)}}}_{L^2_y}
\les 1  
\label{eq:strangely:good}
\end{align}
which follows similarly to \eqref{theta:rho:bound}, and obtain
\begin{align}
\abs{T_{\Bcal}^{(2,2)}} 
&\leq \int_{\RR}   \int_0^\infty |\xi|^{3/2} \abs{y \rho^2 \left(\theta_\xi -1\right)}  |\xi|^{1/2}  \abs{A_\xi} \abs{\frac{\overline{\bar w_\xi}}{y}}  e^{2\tau|\xi|} \brak{\xi}^{2r} \ud y   \ud \xi
\notag\\
&\les  \norm{\abs{\p_x}^{1/2} A}_{\tau,r} \norm{\p_y \bar w}_{\tau,r} \,.
\label{eq:Magic:2:2}
\end{align}
For the third term on the right side of \eqref{eq:Magic:2:0} we simply appeal to \eqref{theta:rho:bound} and \eqref{eq:v:L:infty} to obtain
\begin{align}
\abs{T_{\Bcal}^{(2,3)}} 
&\les \int_{\RR} |\xi| \norm{\rho (\theta_\xi-1)}_{L^2_y} \abs{I_\infty[\bar w_\xi]} \norm{\rho \bar w_\xi}_{L^2_y} e^{2\tau|\xi|} \brak{\xi}^{2r} \ud\xi 
\notag\\
&\les  \int_{\RR} |\xi|^{1/2}  \norm{\rho \bar w_\xi}_{L^2_y}^2 e^{2\tau|\xi|} \brak{\xi}^{2r} d\xi 
\notag\\
&\les \norm{|\p_x|^{1/2} \bar w}_{\tau,r} \norm{\bar w}_{\tau,r} \,.
\label{eq:Magic:2:3}
\end{align}
Lastly, for the nonlinear term in \eqref{eq:Magic:2:0} we similarly have
\begin{align}
\abs{T_{\Bcal}^{(2,4)}} 
&\les \int_{\RR} \int_{\RR} |\xi-\eta| \norm{\rho (\theta_\xi-1)}_{L^2_y} \abs{A_\eta} \abs{A_{\xi-\eta}} \norm{\rho \bar w_\xi}_{L^2_y} e^{2\tau|\xi|} \brak{\xi}^{2r} \ud \eta \ud\xi 
\notag\\
&\les  \int_{\RR} |\xi-\eta| \left(\brak{\eta}^{r-1/2} + \brak{\xi-\eta}^{r-1/2}\right)  \abs{A_\eta} \abs{A_{\xi-\eta}}  \norm{\rho \bar w_\xi}_{L^2_y} e^{2\tau|\xi|} \brak{\xi}^{r+1/2} d\xi 
\notag\\
&\les  \norm{\brak{\p_x}^{1/2} \bar w}_{\tau,r}  \norm{|\p_x|^{1/2} A}_{\tilde{\tau,r}} \norm{A}_{\tilde{\tau,r}}\,.
\label{eq:Magic:2:4}
\end{align}
Combined, the bounds \eqref{eq:Magic:2:1}--\eqref{eq:Magic:2:4} yield an estimate for the $T_{\Bcal}^{(2)}$ term. Next, we estimate the remaining five terms on the right side of \eqref{eq:Magic:0}.

In order to bound $T_{\Bcal}^{(1)}$ 
we note from \eqref{eq:theta:def} that  
\begin{align*}
\abs{\partial_t \theta_\xi} 
&=  \frac{y^2 \brak{\xi}^2}{2 \eps (1+t/\eps)^2} e^{-\frac{y^2 \brak{\xi}^2}{2(1+t/\eps)}} 
\leq \frac{1}{\eps(1+t/\eps)} e^{-\frac{y^2 \brak{\xi}^2}{3(1+t/\eps)}} \les \frac{1}{\eps} e^{-\frac{y^2 \brak{\xi}^2}{3(1+t/\eps)}} 
\end{align*}
and
\begin{align*}
\abs{y \partial_{yy} \theta_\xi} 
\leq  \frac{y \brak{\xi}^2}{(1+t/\eps) } e^{-\frac{y^2 \brak{\xi}^2}{2(1+t/\eps)}} 
+ \frac{y^3 \brak{\xi}^4}{(1+t/\eps)^2} e^{-\frac{y^2 \brak{\xi}^2}{2(1+t/\eps)}}
\les  \brak{\xi} e^{-\frac{y^2 \brak{\xi}^2}{3(1+t/\eps)}} 
\end{align*}
from which we deduce
\begin{subequations}
\begin{align}
\|\rho \p_t \theta_\xi \|_{L^2_y} & \les \frac{1}{\eps  \brak{\xi}^{1/2}}      \label{theta:est:2}
\\
\|y \rho^2  \p_{yy} \theta_\xi \|_{L^2_y} &\les \brak{\xi}^{1/2}    \label{theta:est:1} 
\, .
\end{align}
\end{subequations}
With estimates \eqref{theta:est:1}--\eqref{theta:est:2}, the Hardy inequality and the fact that $\rho \geq 1$, we may estimate 
\begin{align*}
\abs{T_{\Bcal}^{(1)}} 
&\leq \int_{\RR} \abs{A_\xi} \norm{\rho \partial_t \theta_\xi}_{L^2_y} \norm{\rho \bar w_\xi}_{L^2_y} e^{2\tau|\xi|} \brak{\xi}^{2r} \ud \xi + \int_{\RR} \abs{A_\xi} \norm{y \rho^2 \partial_{yy} \theta_\xi}_{L^2_y} \norm{\frac{\bar w_\xi}{y}}_{L^2_y} e^{2\tau|\xi|} \brak{\xi}^{2r} \ud \xi
\notag \\
&\les   \norm{\bar w}_{\tau,r} \norm{A}_{\tilde{\tau,r}} +  \norm{\p_y \bar w}_{\tau,r} \norm{\brak{\p_x}^{1/2} A}_{\tilde{\tau,r}}  
\, .
\label{eq:Magic:1}
\end{align*}

In order to estimate the $T_{\Bcal}^{(3)}$ term in \eqref{eq:Magic:0}, we note that upon applying Lemma~\ref{lem:Hardy} to $ f = I_\infty[\bar w_\xi] - I_y[\bar w_\xi]$, we have that 
\begin{align}
\norm{\rho(y) \left(I_\infty[\bar w_\xi] - I_y[\bar w_\xi]\right)}_{L^2_y} 
&\les  \norm{\rho \bar w_\xi}_{L^2_y}
\end{align}
and thus
\begin{align}
\abs{T_{\Bcal}^{(3)}} 
\les  \int_{\RR} |\xi|  \norm{\rho \bar w_\xi}_{L^2_y}^2 e^{2\tau|\xi|} \brak{\xi}^{2r} \ud \xi
\les  \norm{\abs{\p_x}^{1/2} \bar w}_{\tau,r}^2\,.
\label{eq:Magic:3}
\end{align}

For the $T_{\Bcal}^{(4)}$ term in \eqref{eq:Magic:0}, we use \eqref{eq:strangely:good} to obtain that $\norm{y \rho (\theta_\xi-1)}_{L^2_y} \les (1+t/\eps)^{3/4} \brak{\xi}^{-3/2}$, and the estimate
\begin{align}
\norm{\rho (I_\infty[1-\theta_\xi] - I_y[1-\theta_\xi])}_{L^2_y} 
\les \norm{\int_y^\infty e^{-\frac{(y')^2 \brak{\xi}^2}{2(1+t/\eps)}} \ud y'}_{L^2_y} 
\les \frac{1}{\brak{\xi}^{3/2}}
\label{eq:PatriotsCheat} 
\end{align}
to conclude
\begin{align}
\abs{T_{\Bcal}^{(4)}} 
&\leq \int_{\RR} |\xi| \abs{A_\xi} \left( \norm{\rho y (1-\theta_\xi)}_{L^2_y} + \norm{\rho  (I_\infty[1-\theta_\xi] - I_y[1-\theta_\xi])}_{L^2_y} \right) \norm{\rho \bar w_\xi}_{L^2_y} e^{2\tau|\xi|} \brak{\xi}^{2r} \ud \xi 
\notag \\
&\les    \norm{\bar w}_{\tau,r} \norm{A}_{\tilde{\tau,r}} 
\, .
\label{eq:Magic:4}
\end{align}

Lastly, we turn to the two nonlinear terms in \eqref{eq:Magic:0}.  First, we use \eqref{theta:rho:bound} to estimate 
\begin{align*}
\norm{\rho \abs{\theta_\eta\theta_{\xi-\eta}-1}}_{L^2_y}
\leq \norm{\rho\abs{\theta_{\xi-\eta}-1}}_{L^2_y} + \norm{\rho \abs{\theta_\eta-1}}_{L^2_y} \les  \frac{1}{\brak{\xi-\eta}^{1/2}} +  \frac{1}{\brak{\eta}^{1/2}}
\end{align*}
and then use the bound
\begin{align*}
\brak{\xi}^{r} |\xi|^{1/2}  
\les \brak{\eta}^{r} |\eta|^{1/2}   + \brak{\xi-\eta}^{r} |\xi-\eta|^{1/2}  
\end{align*}
to conclude
\begin{align}
\abs{T_{\Bcal}^{(5)}} 
&\leq \int_{\RR} \int_{\RR} |\xi|^{1/2} \brak{\xi}^r \abs{A_\eta} \abs{A_{\xi-\eta}} \norm{\rho (\theta_{\eta} \theta_{\xi-\eta} -1)}_{L^2_y}  |\xi|^{1/2} \brak{\xi}^r \norm{\rho \bar w_\xi}_{L^2_y} e^{2\tau|\xi|}  \ud \eta \ud \xi 
\notag \\
&\les   \norm{|\p_x|^{1/2} \bar w}_{\tau,r} \norm{A}_{\tilde{\tau,r}}  \norm{|\p_x|^{1/2} A}_{\tilde{\tau,r}}
\, .
\label{eq:Magic:5}
\end{align}
For the last term in \eqref{eq:Magic:0}, we use that 
\begin{align*}
\norm{\rho\, I_y[\theta_\eta] \p_y \theta_{\xi-\eta}}_{L^2_y} \leq \norm{\rho y \p_y \theta_{\xi-\eta}}_{L^2_y} \les \norm{e^{\frac{y^2}{8(1+t/\eps)}} e^{-\frac{y^2 \brak{\xi-\eta}^2}{4(1+t/\eps)}}}_{L^2_y} \les \frac{1}{\langle \xi -\eta\rangle^{1/2}}
\end{align*}
and the triangle inequality $|\eta|^{1/2} \leq |\xi-\eta|^{1/2} + |\xi|^{1/2}$ to conclude that 
\begin{align}
\abs{T_{\Bcal}^{(6)}} 
&\leq \int_{\RR} \int_{\RR} |\eta|^{1/2} (|\xi-\eta|^{1/2} + |\xi|^{1/2}) \brak{\xi}^r \abs{A_\eta} \abs{A_{\xi-\eta}} \norm{\rho I_y[\theta_{\eta}] \p_y \theta_{\xi-\eta}}_{L^2_y}  \brak{\xi}^r \norm{\rho \bar w_\xi}_{L^2_y} e^{2\tau|\xi|}  \ud \eta \ud \xi 
\notag \\
&\les   \norm{\brak{\p_x}^{1/2} \bar w}_{\tau,r} \norm{A}_{\tilde{\tau,r}}  \norm{|\p_x|^{1/2} A}_{\tilde{\tau,r}}
\, .
\label{eq:Magic:6}
\end{align}
Upon collecting the bounds \eqref{eq:Magic:2:1}--\eqref{eq:Magic:2:4}, \eqref{eq:Magic:1}--\eqref{eq:Magic:3}, and \eqref{eq:Magic:4}--\eqref{eq:Magic:6}, we conclude the proof of \eqref{eq:Magic:*}.

\subsection{Error terms in the $\p_y \bar w$ energy}
The estimates in this section are very similar to those in Section~\ref{sec:bar:w:energy}, however, several modifications are in order: we are testing the equation with the conjugate of $\p_y  \bar w$, $\brak{\xi}^{2r}$ becomes $\brak{\xi}^{2r-1}$, and we may use that the cutoff $\chi$ vanishes near $\{ y=0\}$.

\subsubsection{The $T_{\p_y \Ncal}$ term}
From \eqref{eq:dy:nonlinear:terms:a} we see that $\p_y \Ncal_\xi(\bar w,\bar w) = \Ncal_\xi(\bar w,\p_y \bar w)$, and thus the estimates are very similar to the $T_{\Ncal}$ term. From \eqref{eq:dy:nonlinear:terms:a} and \eqref{eq:Kobe:1:a} we have
\begin{align*}
\abs{ T_{\p_y \Ncal} }
&\les \int_\RR \int_\RR \norm{\rho \p_y \bar w_\eta}_{L^2_y} \norm{\chi \rho \p_y \bar w_{\xi-\eta}}_{L^2_y} \norm{\chi \rho \p_y \bar w_\xi}_{L^2_y} \notag\\
&\qquad \qquad\qquad \qquad  \times \abs{\xi-\eta} ( \brak{\eta}^{r-1/2} + \brak{\xi-\eta}^{r-1/2}) \brak{\xi}^{r-1/2} e^{2\tau|\xi|} \ud \eta \ud \xi
\notag\\
&\quad + \int_\RR \int_\RR \norm{\rho  \bar w_\eta}_{L^2_y} \norm{\chi \rho \p_{yy} \bar w_{\xi-\eta}}_{L^2_y} \norm{\chi \rho \p_y \bar w_\xi}_{L^2_y} \notag\\
&\qquad \qquad\qquad \qquad  \times \abs{\eta} ( \brak{\eta}^{r-1/2} + \brak{\xi-\eta}^{r-1/2}) \brak{\xi}^{r-1/2} e^{2\tau|\xi|} \ud \eta \ud \xi
\notag\\
&\les \norm{\brak{\p_x}^{1/2} \chi \p_y \bar w}_{\tau,r-1/2}^2 \norm{\p_y \bar w}_{\tau,r}
\notag\\
&\quad +   \norm{\chi \p_{yy} \bar w}_{\tau,r-1/2}  \norm{\chi \p_{y} \bar w}_{\tau,r-1/2}   \norm{\brak{\p_x}^{1/2} \bar w}_{\tau,r}
\end{align*}
where we have used that $r>2$. The above estimate gives the proof of \eqref{eq:Barkley:@}.

\subsubsection{The $T_{\p_y \Lcal}$ term}
The term $T_{\p_y \Lcal}$ is defined via \eqref{eq:dy:nonlinear:terms:b} and \eqref{eq:Kobe:1:b}, and may be split into four terms $T_{\p_y \Lcal}^{(j)}$ with $j\in \{1,\ldots,4\}$, according to the four terms in the integrand of \eqref{eq:dy:nonlinear:terms:b}. For the first term we have
\begin{align}
\abs{T_{\p_y \Lcal}^{(1)}}
&\les \int_\RR \int_{\RR} \norm{\chi \rho \p_y \bar w_\eta}_{L^2_y} \abs{A_{\xi-\eta}} \norm{\chi \rho \p_y \bar w_\xi}_{L^2_y} |\xi-\eta| \brak{\xi}^{2r-1} e^{2\tau|\xi|} \ud \eta \ud \xi
\notag\\
&\les \norm{\chi \p_y \bar w}_{\tau,r-1/2}^2 \norm{|\p_x|^{1/2} A}_{\tilde{\tau,r}}
\, .
\label{eq:Jordan:@:1}
\end{align}
Using  $\chi\equiv 0$ on $[0,1]$, we have that $\chi(y) \les y^2 \chi(y)$, which we may combine with the pointwise bound 
$
\abs{y^2 \p_y \theta_\xi} \les   \brak{\xi}^{-1},
$
to estimate
\begin{align}
\abs{T_{\p_y \Lcal}^{(2)}}
&\les \int_\RR \int_{\RR} \norm{ \rho  \bar w_{\eta}}_{L^2_y} \norm{\chi y^2 \p_y \theta_{\xi-\eta}}_{L^\infty_y} \abs{A_{\xi-\eta}} \norm{\chi \rho \p_y \bar w_\xi}_{L^2_y} |\xi-\eta| \brak{\xi}^{2r-1} e^{2\tau|\xi|} \ud \eta \ud \xi
\notag\\
&\les  \norm{\chi \p_y \bar w}_{\tau,r-1/2} \norm{\bar w}_{\tau,r} \norm{A}_{\tilde{\tau,r}}
\,.
\label{eq:Jordan:@:2}
\end{align}
For the third term in the definition of $T_{\p_y \Lcal}$ we have
\begin{align}
\abs{T_{\p_y \Lcal}^{(3)}}
&\les \int_\RR \int_{\RR} \norm{\chi \rho \p_y \bar w_{\xi-\eta}}_{L^2_y} \abs{A_{\eta}} \norm{\chi \rho \p_y \bar w_\xi}_{L^2_y} |\xi-\eta| \brak{\xi}^{2r-1} e^{2\tau|\xi|} \ud \eta \ud \xi
\notag\\
&\les \norm{\chi |\p_x|^{1/2} \p_y \bar w}_{\tau,r-1/2}  \norm{\chi \brak{\p_x}^{1/2} \p_y \bar w}_{\tau,r-1/2} \norm{A}_{\tilde{\tau,r}}
\,.
\label{eq:Jordan:@:3}
\end{align}
For the last term, we use again that $\chi(y) \les y^2 \chi(y)$, and that similarly to \eqref{theta:est:1} (and the equation two lines above it) we have 
\[
\abs{\rho^2 y^2 \p_{yy} \theta_\xi} \les \Big(y^2 \brak{\xi}^2 + y^4 \brak{\xi}^4 \Big) e^{\frac{y^2}{4(1+t/\eps)}} e^{-\frac{y^2 \brak{\xi}^2}{2(1+t/\eps)}} \les e^{-\frac{y^2 (\brak{\xi}^2-1)}{4(1+t/\eps)}} \les 1
\]
which implies
\begin{align}
\abs{T_{\p_y \Lcal}^{(4)}}
&\les \int_\RR \int_{\RR} \norm{I_y[\bar w_{\eta}]}_{L^\infty_y} \abs{A_{\xi-\eta}} \norm{\rho^{-1}}_{L^2_y} \norm{\chi y^2 \rho^2 \p_{yy} \theta_{\xi-\eta}}_{L^\infty_y} \norm{\chi \rho \p_y \bar w_\xi}_{L^2_y} |\eta| \brak{\xi}^{2r-1} e^{2\tau|\xi|} \ud \eta \ud \xi
\notag\\
&\les  \int_\RR \int_{\RR} \norm{\rho \bar w_{\eta}}_{L^2_y} \abs{A_{\xi-\eta}}  \norm{\chi \rho \p_y \bar w_\xi}_{L^2_y} |\eta| \brak{\xi}^{2r-1} e^{2\tau|\xi|} \ud \eta \ud \xi
\notag\\
&\les  \norm{\chi \p_y \bar w}_{\tau,r-1/2} \norm{\abs{\p_x}^{1/2} \bar w}_{\tau,r} \norm{A}_{\tilde{\tau,r}} \,.
\label{eq:Jordan:@:4}
\end{align}
Combining the  bounds  \eqref{eq:Jordan:@:1}, \eqref{eq:Jordan:@:2}, \eqref{eq:Jordan:@:3}, and \eqref{eq:Jordan:@:4}, we obtain the proof of \eqref{eq:Jordan:@}. 

\subsubsection{The $T_{\p_y \Mcal}$ term}
The $T_{\p_y \Mcal}$ terms is defined via \eqref{eq:dy:nonlinear:terms:c} and \eqref{eq:Kobe:1:c} as $T_{\p_y \Mcal}^{(1)} + T_{\p_y \Mcal}^{(2)}$, where the decomposition is between the two terms in the integrand of \eqref{eq:dy:nonlinear:terms:c}. For the first term we use that $\abs{\theta_\eta} \leq 1$ to obtain
\begin{align}
\abs{T_{\p_y \Mcal}^{(1)}}
&\les \int_{\RR} \int_{\RR}    |\eta| |A_\eta|   \norm{\chi \rho \p_y \bar w_{\xi-\eta}}_{L^2_y} \norm{\chi \rho \p_y \bar w_{\xi}}_{L^2_y} \brak{\xi}^{2r-1} e^{2\tau |\xi|}  \ud \eta \ud \xi
\notag\\
&\les \norm{\chi \p_y \bar w}_{\tau,r-1/2}^2 \norm{\abs{\p_x}^{1/2} A}_{\tilde{\tau,r}}
\label{eq:Jordan:@:5}
\end{align}
while for the second term we appeal to $\abs{\theta_\eta} \leq y$, which gives
\begin{align}
\abs{T_{\p_y \Mcal}^{(2)}}
&\les \int_{\RR} \int_{\RR}    |\eta| |A_\eta|   \norm{\chi \rho \p_{yy} \bar w_{\xi-\eta}}_{L^2_y} \norm{y \chi \rho \p_y \bar w_{\xi}}_{L^2_y} \brak{\xi}^{2r-1} e^{2\tau |\xi|}  \ud \eta \ud \xi
\notag\\
&\les \norm{y \chi \p_y \bar w}_{\tau,r-1/2} \norm{\chi \p_{yy} \bar w}_{\tau,r-1/2} \norm{\abs{\p_x}^{1/2} A}_{\tilde{\tau,r}}
\, .
\label{eq:Jordan:@:6}
\end{align}
Combining \eqref{eq:Jordan:@:5} and \eqref{eq:Jordan:@:6} we obtain the proof of \eqref{eq:Jordan:@:*}.

\subsubsection{The $T_{\p_y \Bcal}$ term}
According to \eqref{eq:dy:nonlinear:terms:d} and \eqref{eq:Kobe:1:d} we write 
\[
T_{\p_y \Bcal} = \sum_{j=1}^{6} T_{\p_y \Bcal}^{(j)},
\]
where the decomposition is according to the six terms in \eqref{eq:dy:nonlinear:terms:d}. 

Since we are doing estimates on the support of $\chi$, i.e. for $y\geq 1$, $y$-derivatives of the lift function $\theta_\xi$ can be made arbitrarily small on this region, resulting in simpler estimates.
For instance, similarly to \eqref{theta:est:2}--\eqref{theta:est:1}, we may show that 
\begin{align*}
\|\chi \rho^2 \p_t \p_y \theta_\xi \|_{L^\infty_y} &\les \|  \rho^2 y \p_t \p_y \theta_\xi \|_{L^\infty_y}\les \frac{1}{\eps }   
\\
\|\chi \rho^2 \p_{yyy} \theta_\xi \|_{L^\infty_y} &\les \|  \rho^2 y^3 \p_{yyy} \theta_\xi \|_{L^\infty_y} \les 1
\end{align*}
and therefore
\begin{align}
\abs{T_{\p_y \Bcal}^{(1)}}
&\les \int_{\RR} \abs{A_\xi} \|\chi \rho^2 \p_t \p_y \theta_\xi \|_{L^\infty_y} \norm{\rho^{-1}}_{L^2_y} \norm{\chi \rho \p_y \bar w_\xi}_{L^2_y} \brak{\xi}^{2r-1} e^{2\tau |\xi|}  \ud \xi \notag\\
&\qquad +  \int_{\RR} \abs{A_\xi} \|\chi \rho^2 \p_{yyy} \theta_\xi \|_{L^\infty_y} \norm{\rho^{-1}}_{L^2_y} \norm{\chi \rho \p_y \bar w_\xi}_{L^2_y} \brak{\xi}^{2r-1} e^{2\tau |\xi|}  \ud \xi \notag\\
&\les  \frac{1}{\eps }   \norm{A}_{\tilde{\tau,r}} \norm{\chi \p_y\bar w}_{\tau,r-1/2}
\label{eq:Magic:@:1}
\end{align}
where we have used that $\eps \les 1$ and $t\les 1$. We may also directly estimate
\begin{align}
\abs{T_{\p_y \Bcal}^{(3)}}
\les \int_{\RR} \abs{\xi}   \norm{\rho \bar w_\xi}_{L^2_y}  \norm{\chi \rho \p_y \bar w_\xi}_{L^2_y} \brak{\xi}^{2r-1} e^{2\tau |\xi|}  \ud \xi 
\les   \norm{\chi \p_y\bar w}_{\tau,r-1/2} \norm{\abs{\p_x}^{1/2} \bar w}_{\tau,r}
\label{eq:Magic:@:3}
\end{align}
and similarly to \eqref{eq:Magic:@:1} we have
\begin{align}
\abs{T_{\p_y \Bcal}^{(4)}}
&\les \int_{\RR} \abs{A_\xi} \abs{\xi} \|\chi \rho^2 y \p_y \theta_\xi \|_{L^\infty_y} \norm{\rho^{-1}}_{L^2_y} \norm{\chi \rho \p_y \bar w_\xi}_{L^2_y} \brak{\xi}^{2r-1} e^{2\tau |\xi|}  \ud \xi \notag\\
&\les    \norm{\abs{\p_x}^{1/2} A}_{\tilde{\tau,r}} \norm{\chi \p_y\bar w}_{\tau,r-1/2}
\,.
\label{eq:Magic:@:4}
\end{align} 
It remains to treat the $\partial_t A_\xi$ and the nonlinear terms in \eqref{eq:dy:nonlinear:terms:d}.

For the $\partial_t A_\xi$ contribution, namely $T_{\p_y \Bcal}^{(2)}$, we need to use a decomposition that is analogous to \eqref{eq:Magic:2:0}. The main difference is that the $\theta_\xi-1$ are now replaced by $\p_y \theta_\xi$, and as mentioned earlier, $\chi(y) \abs{\p_y \theta_\xi}$ can be made arbitrarily small. In particular, we may use the bound
\begin{align}
\brak{\xi}^j \norm{\chi \rho^2 \p_y \theta_\xi}_{L^\infty_y} \les \brak{\xi}^j  \norm{  \rho^2 y^{1+j} \p_y \theta_\xi}_{L^\infty_y} \les 1,
\label{eq:Celtics:no:good}
\end{align} 
combined with \eqref{eq:A:evo} to estimate
\begin{align}
\abs{T_{\Bcal}^{(2)}}
&\les \int_{\RR}  \abs{\partial_t A_\xi} \norm{\chi \rho^2 \p_y\theta_\xi}_{L^\infty_y} \norm{\rho^{-1}}_{L^2_y} \norm{\chi \rho \p_y \bar w_\xi}_{L^2_y} \brak{\xi}^{2r-1} e^{2\tau |\xi|}  \ud \xi 
\notag\\
&\les  \int_{\RR}  \abs{\xi} c_{\theta,\xi}  \abs{A_\xi} \norm{\chi \rho \p_y \bar w_\xi}_{L^2_y} \brak{\xi}^{2r-1} e^{2\tau |\xi|}  \ud \xi 
+  \int_{\RR}  \frac{1}{\brak{\xi}}   \abs{\xi}^2 \abs{A_\xi} \norm{\chi \rho \p_y \bar w_\xi}_{L^2_y} \brak{\xi}^{2r-1} e^{2\tau |\xi|}  \ud \xi 
\notag\\
&\qquad + \int_{\RR}  \frac{\abs{\xi}}{\brak{\xi}^{1/2}}    \abs{I_\infty[\bar w_\xi]} \norm{\chi \rho \p_y \bar w_\xi}_{L^2_y} \brak{\xi}^{2r-1} e^{2\tau |\xi|}  \ud \xi
\notag\\
&\qquad + \int_{\RR}\int_{\RR}  \frac{|\xi|}{\brak{\xi}}  \abs{A_\eta} \abs{A_{\xi-\eta}} \norm{\chi \rho \p_y \bar w_\xi}_{L^2_y} \brak{\xi}^{2r-1} e^{2\tau |\xi|}  \ud \eta \ud \xi
\notag\\
&\les  \norm{\brak{\p_x}^{1/2} A}_{\tilde{\tau,r}} \norm{\chi \p_y \bar w}_{\tau,r-1/2} + \norm{\bar w}_{\tau,r} \norm{\chi \p_y \bar w}_{\tau,r-1/2}   + \norm{A}_{\tilde{\tau,r}}^2 \norm{\chi \p_y \bar w}_{\tau,r-1/2}
\, .
\label{eq:Magic:@:2}
\end{align}

Lastly, for the two nonlinear contributions, arising due to the the last two terms in \eqref{eq:dy:nonlinear:terms:d}, we again appeal to \eqref{eq:Celtics:no:good} and estimate
\begin{align}
\abs{T_{\Bcal}^{(5)}}
&\les
\int_{\RR}\int_{\RR}   |\xi-\eta| \abs{A_\eta} \abs{A_{\xi-\eta}} \norm{\chi \rho^2 \p_y \theta_{\xi-\eta}}_{L^\infty_y} \norm{\rho^{-1}}_{L^2_y} \norm{\chi \rho \p_y \bar w_\xi}_{L^2_y} \brak{\xi}^{2r-1} e^{2\tau |\xi|}  \ud \eta \ud \xi
\notag\\
&\les \int_{\RR}\int_{\RR}   \abs{A_\eta} \abs{A_{\xi-\eta}}   \norm{\chi \rho \p_y \bar w_\xi}_{L^2_y} \brak{\xi}^{2r-1} e^{2\tau |\xi|}  \ud \eta \ud \xi
\notag\\
&\les  \norm{A}_{\tilde{\tau,r}}^2 \norm{\chi \p_y \bar w}_{\tau,r-1/2}
\label{eq:Magic:@:5}
\end{align}
and
\begin{align}
\abs{T_{\Bcal}^{(6)}}
&\les
\int_{\RR}\int_{\RR}   |\eta| \abs{A_\eta} \abs{A_{\xi-\eta}} \norm{\chi \rho^2 y \p_{yy} \theta_{\xi-\eta}}_{L^\infty_y} \norm{\rho^{-1}}_{L^2_y} \norm{\chi \rho \p_y \bar w_\xi}_{L^2_y} \brak{\xi}^{2r-1} e^{2\tau |\xi|}  \ud \eta \ud \xi
\notag\\
&\les  \int_{\RR}\int_{\RR} |\eta|  \abs{A_\eta} \abs{A_{\xi-\eta}}   \norm{\chi \rho \p_y \bar w_\xi}_{L^2_y} \brak{\xi}^{2r-1} e^{2\tau |\xi|}  \ud \eta \ud \xi
\notag\\
&\les \norm{\abs{\p_x}^{1/2} A}_{\tilde{\tau,r}}  \norm{A}_{\tilde{\tau,r}}  \norm{\chi \p_y \bar w}_{\tau,r-1/2} \, .
\label{eq:Magic:@:6}
\end{align}
By summing the bounds \eqref{eq:Magic:@:1}--\eqref{eq:Magic:@:4}, and \eqref{eq:Magic:@:2}--\eqref{eq:Magic:@:6}, we obtain the bound \eqref{eq:Magic:@}.

\appendix
\section{Review of the incompressible Triple Deck over a flat 2D plate}
\label{sec:appendix}
We roughly follow the presentation from~\cite{Smith82}. We first introduce the variables 
\begin{align} \label{scalings}
  X =  \frac{x-1}{\nu^{3/8}}, \qquad \bar Y = \frac{y}{\nu^{1/2}}, \qquad  Y = \frac{y}{\nu^{5/8}},  \qquad  \tilde Y = \frac{y}{\nu^{3/8}}, \qquad  T = \frac{t}{\nu^{1/4}}  \, .
\end{align}
Here $\bar{Y}$, $Y$, and $\tilde{Y}$ are the fast vertical variables in the main deck, lower deck, and upper deck, respectively. The $X$ variable is the fast variable in the vicinity of the trailing edge, situated at $x=1, y=0$. 
On the fast time scale $T$, to leading order only perturbances in the lower deck are active, while the other decks the fast time dependence does not enter the momentum equation. Throughout this section we abuse notation and write $\nu$ instead of an inverse Reynolds number, i.e. we treat $\nu$ as if it is dimensionless.

\subsection{Main deck}
The ansatz on the solution of the 2D Navier-Stokes equation in this region is 
\begin{align}
(u_M,v_M,p_M) = \left( U_B(\bar Y) + \nu^{\frac 1 8}   u_1(X,\bar  Y, T), \nu^{\frac 1 4}   v_1(  X,\bar  Y, T), \nu^{\frac 1 4}   p_1(X,\bar  Y,T) \right) + \mbox{lower order terms}. 
\label{eq:main:deck:ansatz}
\end{align}

\noindent Above, $U_B$ is defined to be the Blasius boundary layer, introduced in \eqref{Blasius.a} as $U_B:= f'(\frac{\bar{Y}}{\sqrt{x}})$, where $f$ solves \eqref{Blasius.b}. To ease notation, we suppress the $x$ dependence of $U_B$, and denote by $U_B' = \p_{\bar{Y}}U_B$, since our scaling ensures that we are very close to $x = 1$. Inserting ansatz \eqref{eq:main:deck:ansatz} into the 2D Navier-Stokes equations and collecting only the leading order terms we obtain the {\em inviscid type system}
\begin{subequations}
\label{eq:main:deck}
\begin{align}
U_B \partial_{X}   u_1 +   v_1  U_B' &= 0\\
\partial_{\bar Y} p_1 &= 0 \\
\partial_{X}   u_1 + \partial_{\bar Y}   v_1 &= 0
\, .
\end{align}
\end{subequations}
Note that both the time derivative and the dissipation term in the tangential momentum equation drop  out, as they are lower order in $\nu$.
The system \eqref{eq:main:deck} has as a solution
\begin{align}
  u_1= A(X,T) U_B'(\bar Y), \qquad   v_1 = - \partial_{X}A(X,T) U_B(\bar Y), \qquad   p_1 = P(X,T) \,,
  \label{eq:main:deck:sol}
\end{align}
for some unknown functions $A(X,T)$ and $P(X,T)$. Note that the solution \eqref{eq:main:deck:sol} satisfies the boundary condition $u_1|_{\bar{Y} \to \infty} = 0$. This type of  matching condition enforces, from \eqref{eq:main:deck:ansatz}, that the horizontal velocity in the main deck  converges rapidly to the ambient Blasius flow, which is what is observed when the boundary layer separates. 

The boundary condition 
\begin{align*}
A(X, T) \rightarrow 0 \text{ as } X \rightarrow -\infty
\end{align*}
ensures that, approaching from the left, i.e. as $ x \to 1^{-}$ the main deck profile matches with the Blasius  boundary layer profile $U_B\left( \frac{\bar Y}{\sqrt{x}} \right)$. Therefore, at the lateral boundary $x=1^-$ in original variables, which is the same as the boundary $X \to -\infty$ in rescaled variables (as $\nu \to 0$).

\subsection{Lower deck}

Notice now that the main deck, \eqref{eq:main:deck}, contributes a non-zero trace onto $\bar{Y} = 0$, which needs to be adjusted. Thus, it does not suffice to take the main deck as the full flow, as it does not satisfy the no-slip boundary condition. This is the purpose of introducing the lower deck. More precisely, we compute 
\begin{align} \n
u_M|_{\bar{Y} \to 0} &=  U_B(\bar{Y}) + \nu^{\frac 1 8} u_1(X, \bar{Y}, T) \sim \bar{Y} U_B'(0) + \nu^{\frac 1 8} A(X,T) U_B'(0) \\ \label{B:trace}
&\sim  \nu^{\frac 1 8} U_B'(0) (\nu^{-\frac 1 8} \bar{Y} + A(X,T)) \sim \nu^{\frac 1 8} U_B'(0) (Y + A(X,T)),
\end{align}

\noindent where we have used the scaling $Y = \nu^{-\frac 1 8} \bar{Y}$ which relates the lower deck scaling and main deck/ Prandtl scaling. This then suggests that in order to correct for the boundary trace, \eqref{B:trace}, we need to seek a lower deck expansion in of magnitude $\nu^{\frac 1 8}$. The ansatz on the solution $(u_L,v_L,p_L)$ of the 2D Navier-Stokes equation in this region is thus
\begin{align}
(u_L,v_L,p_L) = \left(\nu^{\frac 1 8} U (X, Y, T), \nu^{\frac 3 8} V( X, Y, T), \nu^{\frac 1 4}  P( X, Y, T) \right) + \mbox{lower order terms} 
\label{eq:lower:deck:ansatz}
\,.
\end{align}
Inserting ansatz \eqref{eq:lower:deck:ansatz} into the 2D Navier-Stokes equations and collecting only the leading order terms we obtain the {\em boundary layer type system}
\begin{subequations}
\label{eq:lower:deck}
\begin{align}
\partial_{T}   U +   U \partial_{X}   U +   V \partial_{Y}  U + \partial_{X} P - \partial_{Y}^2 U &= 0\\
\partial_{Y} P &= 0 \\
\partial_{X}  U + \partial_{Y} V &= 0
\, .
\end{align}
\end{subequations}
Matching the tangential velocity as $Y\to \infty$ in \eqref{eq:lower:deck:ansatz}, with the tangential velocity as $\bar Y \to 0$ in \eqref{B:trace}, we arrive at the boundary condition (for simplicity, take $U_B'(0) = 1$)
\begin{align}
\lim_{Y \to \infty} \left( U(X, Y, T) -  Y \right)= A(X, T)
\label{eq:lower:deck:BC:1}
\,.
\end{align}
On the other hand, at the boundary of the plate we impose Dirichlet boundary conditions
\begin{align}
U(X,0,T) = V(X,0,T) = 0 \, .
\label{eq:lower:deck:BC:2}
\end{align}

\subsection{Upper deck}
The flow in this region is to leading order of steady potential inviscid type. That is, the leading order is an Euler flow, which takes as argument the unscaled variables, $(t, x, y)$, from \eqref{scalings}. In comparison to the perturbations, which in this deck are functions of $(X, \tilde{Y})$, every Euler flow fluctuates slowly, and so, without loss of generality, we take the outer Euler flow to be the constant shear flow $(1, 0)$. Thiss yields the ansatz
\begin{align}
(u_U,v_U,p_U) = \left(1 + \nu^{\frac 1 4} u_2(X, \tilde Y, T), \nu^{\frac 1 4} v_2( X,\tilde Y, T), \nu^{\frac 1 4} p_2( X,\tilde Y, T) \right) + \mbox{lower order terms} 
\, .
\label{eq:upper:deck:ansatz}
\end{align}
Inserting this ansatz into the 2D Navier-Stokes equations and collecting only the leading order terms we obtain the {\em potential type system} 
\begin{subequations}
\label{eq:upper:deck}
\begin{align}
\partial_X u_2 + \partial_X p_2&= 0\\
\partial_X v_2 + \partial_{\tilde Y} p_2 &= 0 \\
\partial_{X} u_2 + \partial_{\tilde Y} v_2 &= 0
\, .
\end{align}
\end{subequations}
The matching condition at $\tilde Y = 0$ with the flow in the main deck as $\bar Y \to \infty$ requires that
\begin{align}
p_2(X,0,T) = P(X,T), \qquad v_2(X,0,T) = - \partial_X A(X,T) 
\label{eq:upper:deck:BC}
\end{align}
which, according to \eqref{eq:main:deck}, cancels out the normal velocity component from the main deck. Matching the upper deck velocity with the outer Euler solution as $\tilde Y \to \infty$ yields
\begin{align*}
\lim_{\tilde Y \to + \infty} p_2(X,\tilde Y ,T) = \lim_{\tilde Y \to + \infty} v_2(X,\tilde Y ,T) = 0 \, .
\end{align*}

\subsection{The closed coupled system}

From \eqref{eq:upper:deck} we deduce that the pressure $p_2$ and the normal velocity $v_2$ are harmonic in the variables $X, \tilde Y$. But one can say more: the functions $p_2$ and $v_2$ are harmonic conjugates. Therefore, we may view $v_2$ as the real part of an analytic function, and $p_2$ as its imaginary part. Therefore, their traces at the boundary of the half space are related via the Hilbert transform
\begin{align}
P(X,T) = p_2(X,0,T) = H v_2(X,0,T) = - (H \p_X) A(X,T) = \frac{1}{\pi} p.v. \int_{\RR} \frac{\partial_{X'} A(X')}{X - X'} \ud X' 
\label{eq:lower:deck:BC:3}
\end{align}
which concludes the proof of \eqref{eq:TD:A:p}. Recall that $- H \p_X = \abs{\p_X}$. Therefore, the system \eqref{eq:lower:deck} together with the boundary conditions \eqref{eq:lower:deck:BC:1}, \eqref{eq:lower:deck:BC:2}, and \eqref{eq:lower:deck:BC:3} form a closed evolution system, which is called the Triple Deck model. Once the solution in the lower deck is determined, we derive from \eqref{eq:main:deck:sol} the leading order solution in the main deck, while from \eqref{eq:upper:deck:BC} and harmonic extension in the upper half space, we determine the leading order solution in the upper deck.

\subsection{The scalings}

Now that we have presented the derivation of the model, we briefly discuss the idea behind the scalings \eqref{scalings}. We roughly follow the exposition in \cite{Lagree}. \textit{a-priori}, one wants to rescale the $x$ variable near the point of separation, which may physically correspond to the trailing edge of a flat plate, or if a plate were to have a disturbance. In our presentation, $x = 1$ is this point. To achieve this, one introduces a fast, horizontal variable. Next, one scales the magnitude of the deviation from Blasius in the main deck, which represents the separation effect. Summarizing the starting point: 
\begin{align} \label{ans:2}
X = \frac{x - 1}{L}, \qquad u_M = U_B(\bar{Y}) + \ell u_1(X, \bar{Y}, T), \qquad v_M = \frac{\sqrt{\nu} \ell}{L} v_1(X, \bar{Y}, T). 
\end{align}
for scalings $L, \ell$ to be determined in terms of $\nu$. An inspection of \eqref{scalings} shows that $L = \nu^{\frac 3 8}$ and $\ell = \nu^{\frac 1 8}$.

In the next step, one introduces a lower deck to cancel out the boundary contribution from the main deck. This contribution, the content of \eqref{B:trace}, is now of the form 
\begin{align*}
u_M|_{\bar{Y} \to 0} \sim \bar{Y} U_B'(0) + \ell A(X,T) U_B'(0) \sim \ell U_B'(0) (Y + A(X,T)), 
\end{align*}
where the lower deck fast variable and magnitude are now 
\begin{align*}
Y := \ell^{-1} \bar{Y}, \qquad u_L := \ell U.
\end{align*}
In the lower deck step, equating the convection $u_L \p_x$ and the viscosity $\nu \p_{yy}$ just as in the standard Prandtl theory gives the relation 
\begin{align*}
\ell^2 L^{-1} \sim u_L \p_x u_L  = \nu \p_{yy} u_L \sim \nu \frac{1}{\nu} \frac{1}{\ell^2} \ell, 
\end{align*}
which gives $\ell^3 = L$. The final physical determination comes from the upper deck. One realizes that the contribution of $v_M$ at the top of the Main Deck needs to match the order of the pressure, which is $\ell^2$. Thus,
\begin{align*}
\frac{\sqrt{\nu}}{ \ell^2} = \ell^2 \quad \Rightarrow \quad \ell = \nu^{\frac 1 8}, \quad L = \nu^{\frac 3 8} \, ,
\end{align*}
which concludes the scale analysis.

\subsection*{Acknowledgments} 
The authors are grateful to Stephen Childress for pointing them to a number of references concerning the physical origins of the Triple Deck. We also thank Klaus Widmayer for discussions concerning the Benjamin-Ono equation.
The work of S.I.  was   supported  by the NSF grant DMS-1802940.
The work of V.V. was partially supported by the NSF grant DMS-1911413.



\end{document}